\documentclass[11pt,a4paper]{amsart}

\usepackage{amssymb}
\usepackage{amsmath,amsthm,geometry,enumerate}
\usepackage{graphicx}
\usepackage{psfrag}
\usepackage{amscd}
\usepackage[all]{xy}
\usepackage{calligra,mathrsfs}
\usepackage{MnSymbol}
\usepackage{tikz}
\usepackage{booktabs}
\usetikzlibrary{trees}

\tikzstyle{every picture}=[scale=.5,inner sep=10]
\tikzstyle{every node}=[draw,circle,inner sep=1]

\DeclareMathOperator{\pic}{Pic}

\DeclareMathOperator{\ShHom}{\mathscr{H}\text{\kern -3pt {\calligra\large om}}\,}
\DeclareMathOperator{\Aut}{Aut}

\DeclareMathOperator{\spec}{Spec}

\DeclareMathOperator{\codim}{codim}

\DeclareMathOperator{\ch}{ch}
\DeclareMathOperator{\td}{td}
\DeclareMathOperator{\pt}{pt}
\DeclareMathOperator{\calP}{\mathcal{P}}

\newcommand{\Mm}[2]{\mathcal{M}_{#1,#2}}
\newcommand{\Mb}{\overline{\mathcal{M}}}
\newcommand{\Cb}{\overline{\mathcal{C}}}
\newcommand{\C}{{\mathcal{C}}}
\newcommand{\Jb}{\overline{\mathcal{J}}}
\newcommand{\J}{{\mathcal{J}}}
\newcommand{\Tau}{F_{\text{tau}}}

\newcommand{\ThDiv}{\overline{\Theta}}
\newcommand{\ThDivCl}{\overline{\theta}}
\newcommand{\Pic}{\operatorname{Pic}}

\newcommand{\Chow}{\operatorname{A}}

\newtheorem{theorem}{Theorem}[section]
\newtheorem*{theoremNL}{Theorem}
\newtheorem{proposition}[theorem]{Proposition}

\newtheorem*{corollaryNL}{Corollary}
\newtheorem{corollary}[theorem]{Corollary}
\newtheorem{lemma}[theorem]{Lemma}
\theoremstyle{definition}
\newtheorem{definition}[theorem]{Definition}
\theoremstyle{remark}
\newtheorem{remark}[theorem]{\textbf{Remark}}
\newtheorem{example}[theorem]{\textbf{Example}}
\newtheorem{notation}[theorem]{\textbf{Notation}}
\newtheorem{construction}[theorem]{\textbf{Construction}}

\begin{document}

\title{Extensions of the  universal theta divisor}
\author{Jesse Leo Kass}

\address{Current: J.~L.~Kass, Dept.~of Mathematics, University of South Carolina, 1523 Greene Street, Columbia, SC 29208, United States of America}
\email{kassj@math.sc.edu}
\urladdr{http://people.math.sc.edu/kassj/}

\address{Former: J.~L.~Kass, Leibniz Universit\"{a}t Hannover, Institut f\"{u}r algebraische Geometrie, Welfengarten 1, 30060 Hannover, Germany}

\author{Nicola Pagani}

\address{Current: N.~Pagani, Department of Mathematical Sciences, University of Liverpool, Liverpool, L69 7ZL, United Kingdom}
\email{pagani@liv.ac.uk}
\urladdr{http://pcwww.liv.ac.uk/~pagani/}

\address{Former: N.~Pagani, Leibniz Universit\"{a}t Hannover, Institut f\"{u}r algebraische Geometrie, Welfengarten 1, 30060 Hannover, Germany}

\begin{abstract}
The Jacobian varieties of smooth curves fit together to form a family, the universal Jacobian, over the moduli space of smooth pointed curves, and the theta divisors of these curves form a divisor in the universal Jacobian.   In this paper we describe how to extend these families over the moduli space of stable pointed curves using a stability parameter.  We then   prove a wall-crossing formula describing how the theta divisor varies with this parameter.

We use this result to analyze divisors on the moduli space of smooth pointed curves that have recently been studied by Grushevsky--Zakharov, Hain and M\"{u}ller.  Finally, we compute the pullback of the theta divisor studied in Alexeev's work on stable semiabelic varieties and in Caporaso's work on theta divisors of compactified Jacobians.
\end{abstract}
\maketitle

\bibliographystyle{amsalpha}

\setcounter{tocdepth}{2}
\tableofcontents
\section {Introduction}
In this paper, we study the enumerative geometry of  compactified universal  Jacobians, moduli spaces closely related to the  moduli space $\Mb_{g, n}$ of stable pointed curves.  The degree $d$ (uncompactified) universal Jacobian $\J^{d}_{g, n}$ is the moduli space of degree $d$ line bundles on smooth pointed curves, and a compactified universal Jacobian is a compactification defined by including certain line bundles on stable pointed curves together with sheaves that arise as their degenerations.  The study of $\Mb_{g, n}$ has long held a central role in modern enumerative geometry, but the enumerative geometry of the compactified universal Jacobian has only recently been studied.

An important difference between $\Mm{g}{n}$ and $\J^{d}_{g, n}$ is that essentially the only compactification used in the study of the enumerative geometry of $\Mm{g}{n}$ is the compactification $\Mb_{g, n}$ by stable curves, but the universal Jacobian admits many natural compactifications, each of which should play an important role in the enumerative geometry.  These different compactifications arise because the moduli space of all line bundles on stable curves  is badly behaved (for example, it is not separated), and to obtain a well-behaved moduli space one needs to \emph{choose} a stability condition.  It is natural to expect that the enumerative geometry of a compactified universal Jacobian depends on the specific choice of this stability condition, and that the dependence on the choice is an important structural feature.  Here we focus on giving a complete picture for codimension $1$ classes. We consider codimension $1$ classes up to rational equivalence (i.e.~as elements of the Chow group), but the distinction between rational equivalence or other adequate equivalence relations (e.g.~numerical or cohomological equivalence) does not play a significant role in this paper.

Our main results are about the class of the theta divisor.  On the universal Jacobian $\J_{g, n} := \J^{g-1}_{g, n}$ parameterizing degree $g-1$ line bundles on smooth pointed curves, the theta divisor $\Theta \subset \J_{g, n}$ is the codimension $1$ locus parameterizing line bundles with a nonzero global section.  We describe how $\Theta$ extends to a given compactification and  how the class of the extension varies when changing the compactification.  Our description of the theta divisor provides a complete description of how codimension classes $1$ vary, in the sense that the group of these classes is generated by the class of the extended theta divisor and its translates, together with classes pulled back from $\Mb_{g, n}$, and the latter do not vary in an interesting way with the compactification of $\J_{g,n}$.

Our results about the theta divisor illuminate a problem studied by Grushevsky--Zakharov, Hain, and M\"{u}ller which we now recall.   Given a sequence $\vec{d} = (d_1, \dots, d_n)$ of integers with  $\sum d_j = g-1$ and at least one $d_j$ negative, one can consider possible extensions to $\Mb_{g,n}$ of the closed subset
\[
	D_{\vec{d}} := \{ (C, p_1, \dots, p_n) \in \Mm{g}{n} \colon h^{0}(C, \mathcal{O}( d_{1} p_1+\dots+ d_{n} p_n)) \ne 0 \} \subset \Mm{g}{n}.
\]
One way of extending $[D_{\vec{d}}]$ was suggested by Hain in \cite[Section~11.2, page~561]{hain13}.  The rule $(C, p_{1}, \dots, p_{n}) \mapsto \mathcal{O}( d_1 p_1 + \dots + d_{n} p_{n})$ defines a section
\begin{equation} \label{Eqn: SectionMap}
	s_{\vec{d}} \colon \Mm{g}{n} \to \J_{g,n}
\end{equation}
of the family $\J_{g,n} \to \Mm{g}{n}$, with the property that $D_{\vec{d}}$ is the preimage of the theta divisor
\[
	\Theta := \{ (C, p_1, \dots, p_n; F) \colon h^{0}(C, F) \ne 0 \}.
\]
Thus one way to extend $D_{\vec{d}}$ is to extend \eqref{Eqn: SectionMap} to a morphism $\Mb_{g,n} \to \Jb_{g,n}$ (into some extension $\Jb_{g,n}$ of $\J_{g,n}$), to extend $\Theta$ to a divisor $\ThDiv$ on $\Jb_{g,n}$, and then to take the preimage of $\ThDiv$ by ${s}_{\vec{d}}$.  The difficulty in carrying out this idea is that already the moduli space $\widetilde{\J}_{g,n}$, parameterizing line bundles of degree $g-1$ on stable curves, is not separated.    In particular,  while the section \eqref{Eqn: SectionMap} does extend to a morphism $\Mb_{g,n} \to \widetilde{\J}_{g,n}$, there is not a unique extension, an  issue already observed by Hain, who remarks that this is a ``subtle problem'' \cite[Section~11.2, page~561]{hain13}.

The failure for $\widetilde{J}_{g,n}$ to be separated is intimately related to an invariant of a line bundle $L$ on a reducible curve $C$: the multidegree.  The multidegree $\deg(L)(C)$ is defined to be the vector whose components are the degrees of the restrictions of $L$ to the irreducible components of $C$.  To obtain a well-behaved moduli space of line bundles, we impose a stability condition as a numerical condition on the multidegree of a line bundle.  There is now a large body of literature (surveyed in Section~\ref{Subsection: Remarks})  on how to construct a moduli space associated to a stability condition, and we build upon that literature, especially the work \cite{oda79} of Oda--Seshadri, to construct a collection of extensions of $\J_{g,n}(\phi)$ indexed by a linear algebra parameter $\phi$.

We focus on studying extensions of  $\J_{g, n}$  to a family over the moduli space $\Mm{g}{n}^{\text{TL}} \subseteq \Mb_{g, n}$  of treelike curves rather than the moduli space of all stable curves because, for our purposes, these moduli spaces are best suited to studying the theta divisor.  Indeed, a codimension $1$ class on an extension over $\Mb_{g, n}$ is determined by its restriction over $\Mm{g}{n}^{\text{TL}}$ because, quite generally, the Chow class of a divisor is determined by its restriction to the complement of a closed subset of codimension at least $2$.  Furthermore, we prove that the different extensions of $\J_{g,n}$ over $\Mm{g}{n}^{\text{TL}}$ are closely related to the different extensions of $\Theta$, and this relationship becomes less transparent when working over $\Mb_{g,n}$. We discuss this issue at the end of  Section~\ref{Subsection: Remarks}. The restriction to $\Mm{g}{n}^{\text{TL}}$ is not essential to the construction of the compactified Jacobians and the theta divisors, which are in fact defined as restrictions of analogous objects over $\Mb_{g, n}$.

We construct the extensions of $\J_{g,n}$ in Section~\ref{Section: polytopeDecomposition}.  There we construct an affine space $V_{g,n}^{\text{TL}}$, \emph{the stability space}, and for every nondegenerate stability parameter $\phi \in V_{g,n}^{\text{TL}}$ a family of moduli spaces $\Jb_{g, n}(\phi) \to \Mm{g}{n}^{\text{TL}}$ extending the universal Jacobian  $\J_{g,n} \to \Mm{g}{n}$.

The affine space $V_{g,n}^{\text{TL}}$ decomposes into a \emph{stability polytope decomposition}, a decomposition into polytopes such that $\phi_1$-stability coincides with $\phi_2$-stability if and only if $\phi_1$ and $\phi_2$ lie in a common polytope.  For every nondegenerate $\phi \in V_{g,n}^{\text{TL}}$, we construct a divisor $\ThDiv(\phi) \subseteq \Jb_{g, n}(\phi)$ extending $\Theta$ and then we describe how the associated Chow class $\ThDivCl(\phi)$ depends on $\phi$ as follows.  For any two nondegenerate stability parameters $\phi_1$ and $\phi_2$, there is a distinguished isomorphism between the Chow groups $\Chow^{1}(\Jb_{g, n}(\phi_{1})) \cong \Chow^{1}(\Jb_{g, n}(\phi_{2}))$, which allows to form the difference $\ThDivCl(\phi_2) - \ThDivCl(\phi_1)$.

If we further suppose that $\phi_1$ and $\phi_2$ lie on opposite sides of a wall in the stability space $V_{g,n}^{\text{TL}}$, then $\phi_1$-stability coincides with $\phi_2$-stability on every curve except those lying in some boundary divisor $\Delta_{i,S}$.   Here $\Delta_{i,S}$ is the closure of the locus of curves consisting of a genus $i$ curve connected to a genus $g-i$ curve by one node and having the markings $S$ lying on the genus $i$ curve. On such a curve, there is a unique $\phi_1$-stable bidegree of a line bundle, say $(d, g-1-d)$. After possibly switching the roles of $\phi_1$ and $\phi_2$, the only $\phi_2$-stable bidegree is $(d+1, g-2-d)$.   Denoting the (pullback to the compactified universal Jacobian of the) class of $\Delta_{i,S}$ by $\delta_{i,S}$, our main result is

\begin{theoremNL}[Theorem~\ref{Thm: wc}]

	\begin{align} \label{Eqn: WallCrossing}
		\ThDivCl(\phi_2) - \ThDivCl(\phi_1) 	=& \left(d+1-i\right) \cdot \delta_{i,S}.
	\end{align}
\end{theoremNL}

We  use this formula in Section~\ref{PullbackTheta} to study the classes of the different extensions of $D_{\vec{d}}$.  Specifically, for every nondegenerate stability parameter $\phi$ the section in Equation \eqref{Eqn: SectionMap} extends \emph{uniquely} to a morphism
\begin{equation}
	s_{\vec{d}} \colon \Mm{g}{n}^{\text{TL}} \to \Jb_{g, n}(\phi),
\end{equation}
so we can form the preimage
\[
	\overline{D}_{\vec{d}}(\phi) := {s}_{\vec{d}}^{-1}(\ThDiv(\phi))
\]
and prove the following result.

\begin{theoremNL}[Theorem~\ref{pullback}]
For a nondegenerate stability parameter $\phi$, we have
	\begin{equation} \label{wallcrossresult}
		[D_{\vec{d}}(\phi)]= - \lambda + \sum_{j=1}^n {{d_j+1}\choose{2}} \cdot \psi_j + \sum_{(i,S)} \left( {{ d(i,S) - i+1}\choose{2}}- {{ d_S- i+1}\choose{2}} \right) \cdot \delta_{i,S},
	\end{equation}
where $d(i,S)$ is the unique integer such that $(d(i,S), g-1-d(i,S))$ is the bidegree of a $\phi$-stable line bundle on a general element of $\Delta_{i,S} \subset \Mb_{g,n}$.
\end{theoremNL}
Our proof illuminates the structure of this formula. There is a stability parameter $\phi_{\vec{d}}$ such that $\mathcal{O}_C(d_1 p_1 + \ldots + d_n p_n)$ is $\phi_{\vec{d}}$-stable for which
\[
[D_{\vec{d}}(\phi_{\vec{d}})]= - \lambda + \sum_{j=1}^n {{d_j+1}\choose{2}} \cdot \psi_j
\]
and the other terms appearing in Formula \eqref{wallcrossresult} arise by applying the wall-crossing Formula~\eqref{Eqn: WallCrossing}.

A second natural extension of $D_{\vec{d}}$ is the Zariski closure, which we call $\overline{D}_{\vec{d}}(\text{M\"{u}})$ because the corresponding cycle class was computed by M\"uller in  \cite[Theorem~5.6]{mueller13}. A third extension was given by Hain, who extended it to a rational Chow class $[\overline{D}_{\vec{d}}(\text{Ha})]$ using the formalism of theta functions and then computed its class in  \cite[Theorem~11.7]{hain13}. Using different methods, both results were reproved by Grushevsky and  Zakharov  in \cite[Theorem~2, Theorem~6]{grushevsky14a}.

A fourth important extension of $\J_{g,n}$ is given by the theory of degenerate principally polarized abelian varieties.  The family $(\J_{g,n}/ \Mm{g}{n}, \Theta)$ is a family of principally polarized torsors for abelian varieties, and this family uniquely extends to a family $(\overline{\J}_{g,n}/\Mb_{g,n}, \overline{\Theta})$ of stable semiabelic pairs, or stable principally polarized degenerate abelian varieties.  The resulting morphism $s_{\vec{d}} \colon \Mb_{g,n} \dashrightarrow \Jb_{g,n}$ is only a rational morphism, but rational morphisms induce pullback maps on Chow groups, so we can define the extension
\[
	[\overline{D}_{\vec{d}}(\text{SP})] := s_{\vec{d}}^{*}(\ThDivCl).
\]

 We describe the relation between the divisors $[\overline{D}_{\vec{d}}(\phi)]$ and  $[\overline{D}_{\vec{d}}(\text{Ha})]$, $[\overline{D}_{\vec{d}}(\text{SP})]$, $[\overline{D}_{\vec{d}}(\text{M\"{u}})]$ in Sections~\ref{hgz}, \ref{Subsect: StablePair}, \ref{Subsect: Mueller} respectively.
 In particular, we prove the new result
\begin{corollaryNL} (Corollary~\ref{Cor: StablePair}) \label{Corollary: NL}
	The pullback of the theta divisor of the family of stable semiabelic pairs extending $(\J_{g,n}, \Theta)$ satisfies
	\begin{align}
		[\overline{D}_{\vec{d}}(\text{SP})] 	=& [\overline{D}_{\vec{d}}(\phi)] \quad \text{for any nondegenerate $\phi$ such that $\phi_{\text{can}} \in \overline{\mathcal{P}}(\phi)$}\label{Eqn: PairOne} \\
\label{Eqn: PairTwo}						=&  - \lambda + \sum_{j=1}^n {{d_j+1}\choose{2}}\cdot  \psi_j - \sum_{(i,S)} {{d_S -i +1} \choose {2}}\cdot  \delta_{i,S}.
	\end{align}
\end{corollaryNL}
\noindent Here $\overline{\mathcal{P}}(\phi)$ denotes the closure of the unique stability polytope containing $\phi$, and the stability parameter $\phi_{\text{can}}$ is  a distinguished degenerate parameter called the canonical parameter. The corresponding compactified universal Jacobian is the one studied by  e.g.~Caporaso \cite{caporaso08a, caporaso09} and Pandharipande \cite{pand96}.   The stability parameter $\phi_{\text{can}}$ is a vertex of the stability polytope decomposition (Remark~\ref{phicanvertex}), so there are many stability polytopes $\mathcal{P}$ satisfying $\phi_{\text{can}} \in \overline{\mathcal{P}}$.  These stability polytopes play a distinguished role because they are exactly the polytopes $\mathcal{P}$ such that $\ThDiv(\phi) \to \Mm{g}{n}^{\text{TL}}$ is flat for $\phi \in \calP$ (Lemma~\ref{Lemma: WhenIsPhiFlat?}).

After this paper was first posted to the arXiv, the authors were made aware of related work of Bashar Dudin.  In \cite{dudin}, Dudin studies certain divisors, including the theta divisor, on the compactified universal Jacobians constructed by Melo in  \cite{melo16} and computes their pullbacks to  $\Mm{g}{n}^{\text{TL}}$  using techniques similar to the ones used in this papers.  As we explain in Section~\ref{Subsect: StablePair}, one of the moduli spaces studied by Dudin is, in the notation of this paper,  $\Jb_{g,n}(\phi)$ for a certain $\phi$ satisfying $\phi_{\text{can}} \in \overline{\mathcal{P}}$, and he computes the pullback of $\ThDiv(\phi)$ as the class  in Equation~\eqref{Eqn: PairTwo}. The authors first became aware of Dudin's work on July 14, 2015.  The authors first posted their preprint to the arXiv on July 13, 2015 and first publicly presented their work in a seminar on March 10, 2015.  Dudin posted his paper to the arXiv on May 12, 2015.

\section {Conventions}
We work over a fixed algebraically closed field $k$ of characteristic $0$.

A  \textbf{curve} over a field $\spec(K)$ is a $\spec(K)$-scheme $C/\spec(K)$ that is proper over $\spec(K)$, geometrically connected, and pure of dimension $1$. A curve $C/\spec(K)$ is a \textbf{nodal curve} if $C$ is geometrically reduced and the completed local ring of $C \otimes \overline{K}$ at a non-regular point is isomorphic to $\overline{K}[[x,y]]/(xy)$.  Here $\overline{K}$ is an algebraic closure of $K$.

A \textbf{family of curves} over a $k$-scheme $T$ is a proper, flat morphism $C \to T$ whose fibers are curves. A family of curves $C \to T$ is a \textbf{family of nodal curves} if the fibers are nodal curves.

If $F$ is a rank~$1$, torsion-free sheaf on a nodal curve $C$ with irreducible components $ C_i$, then we define the \textbf{multidegree} by ${\deg}(F)(C) := (\deg(F_{C_{i}}))$.  Here $F_{C_{i}}$ is the maximal torsion-free quotient of $F \otimes \mathcal{O}_{C_{i}}$. We define the \textbf{total degree}  to be $\deg_C(F):=\chi(F)-1+p_a(C)$ where $p_a(C)= h^1(C, \mathcal{O}_C)$ is the arithmetic genus of $C$. The total degree and the multidegree of $F$ are related by the formula $\deg_C(F) = \sum \deg_{C_i} F - \delta_C(F)$, where $\delta_C(F)$ denotes the number of nodes of $C$ where $F$ fails to be locally free.

Given a ring $R$ and a set $S$, we write $R^{S}$ for the $R$-module of functions $S \to R$, a free $R$-module with basis indexed by $S$.

We  work with several divisors and their classes. If $\mathcal{X}$ is a smooth Deligne--Mumford stack and $\mathcal{U}$ is an open substack whose complement has codimension $\geq 2$, any Chow class in $\Chow^1(\mathcal{X})$ is completely determined by its restriction to $\mathcal{U}$. For this reason we abuse the notation and denote with the same symbol a Chow class in $\Chow^1(\mathcal{X})$ and in $\Chow^1(\mathcal{U})$. The main examples are the theta divisors introduced in Definition~\ref{Def: Theta}. 

\subsection{Graphs}

A graph $\Gamma$ is a tuple $(\operatorname{Vert}, \operatorname{HalfEdge}, \operatorname{a}, \operatorname{i})$ consisting of a set $\operatorname{Vert}$ that we call the vertex set, a set $\operatorname{HalfEdge}$ that we call the half-edges set, an assignment function $\operatorname{a} \colon \operatorname{HalfEdge} \to \operatorname{Vert}$, and a fixed point free involution $\operatorname{i} \colon \operatorname{HalfEdge} \to \operatorname{HalfEdge}$. The edge set  is defined as the quotient set $\operatorname{Edge}:= \operatorname{HalfEdge} / \operatorname{i}$. The endpoints of an edge $e \in \operatorname{Edge}$ are defined to be $v_1=a(h_1)$ and $v_2=a(h_2)$, where $\{h_1, h_2\}$ is the equivalence class represented by $e$. A loop based at $v$ is an edge whose endpoints both equal $v$.

A $n$-marked graph is a graph $\Gamma$ together with a (genus) map $g \colon \operatorname{Vert}(\Gamma) \to \mathbb{N}$ and a (markings) map $p \colon \{ 1, \dots, n \} \to \operatorname{Vert}(\Gamma)$.  We call $g(v)$ the (geometric) genus of $v \in \operatorname{Vert}(\Gamma)$.  If $v=p(j)$, then we say that the marking $j$ lies on the vertex $v$. A subgraph $\Gamma'$ is always assumed to be proper ($\operatorname{Vert}(\Gamma') \subsetneq \operatorname{Vert}(\Gamma)$) and complete (for all $v_1, v_2 \in \operatorname{Vert}(\Gamma')$, if $h_1, h_2 \in \operatorname{HalfEdge}(\Gamma)$, $\operatorname{a}(h_i) = v_i$ and $\operatorname{i}(h_1)=h_2$, then $h_1, h_2 \in \operatorname{HalfEdge}(\Gamma')$) and it is given the induced genus and marking maps. Because of these assumptions, to define a subgraph $\Gamma' \subset \Gamma$ it is enough to specify its vertex set $\operatorname{Vert}(\Gamma') \subset \operatorname{Vert}(\Gamma)$.

We say that a $n$-marked graph is \textbf{stable} if it is connected, and if for all $v$ with $g(v)=0$, the sum of the number of edges with $v$ as an endpoint plus the number of markings lying on $v$ is at least $3$ (when counting edges, count a loop based at $v$ twice). The (arithmetic) \textbf{genus} of the graph is $g(\Gamma) := \sum_{v \in \operatorname{Vert}(\Gamma)} g(v) -\# \operatorname{Vert}(\Gamma)+  \# \operatorname{Edge}(\Gamma)+1 $.

An \textbf{isomorphism} of $\Gamma= (\operatorname{Vert}, \operatorname{HalfEdge}, \operatorname{a}, \operatorname{i})$ to  $\Gamma'= (\operatorname{Vert}', \operatorname{HalfEdge}', \operatorname{a}', \operatorname{i}')$ is a pair of bijections $\alpha_V \colon \operatorname{Vert} \to \operatorname{Vert}'$ and $\alpha_{\operatorname{HE}} \colon \operatorname{HalfEdge} \to \operatorname{HalfEdge}'$ that satisfy the compatibilities $\alpha_{\operatorname{HE}} \circ i = i'$ and $\alpha_V \circ a = a'$. If $\Gamma$ and $\Gamma'$ are endowed with structures of $n$-marked graphs by the maps $(g,p)$ and by $(g',p')$ respectively, $(\alpha_V, \alpha_{\operatorname{HE}})$ is an isomorphism of $n$-marked graphs if it also satisfies the compatibilities $\alpha_V \circ p = p'$ and $\alpha_V \circ g = g'$. An \textbf{automorphism} is an isomorphism of a graph to itself.

If $\Gamma$ is a $n$-marked graph and $e \in \operatorname{Edge}(\Gamma)$ is an edge,  the \textbf{strict contraction} of $e$ in $\Gamma$ is the graph $\Gamma_e$ where the half-edges corresponding to $e$ are removed, the two (possibly coinciding) endpoints $v_1$ and $v_2$ of $e$ are replaced by a unique vertex $v_e$, and the genus and marking functions are extended to $v_e$ by $p_e(j):=v_e$ whenever $p(j)$ equals $v_1$ or $v_2$, and \[g_e(v_e):= \begin{cases}g(v_1)+g(v_2) & \textrm{ when } e \textrm{ is not a loop},\\ g(v_1)+1 & \textrm{ when } e \textrm{ is a loop.} \end{cases}\]
A stable graph $\Gamma$ is \textbf{treelike} if the graph obtained from $\Gamma$ by strictly contracting all loops is a tree.

If $\Gamma$ and $\Gamma'$ are $n$-marked graphs, a \textbf{contraction}   $c\colon \Gamma \to \Gamma'$ of $e$ in $\Gamma$ is an isomorphism of the strict contraction $\Gamma_e$ to $\Gamma'$. The contraction $c$ is completely determined by the two maps  $c_V \colon \operatorname{Vert}(\Gamma) \to \operatorname{Vert}(\Gamma')$ and $c_{\operatorname{HE}} \colon \operatorname{HalfEdge}(\Gamma) \to \operatorname{HalfEdge}(\Gamma')$ that it induces on vertices and on half-edges respectively.

We will need to take products over certain collections of treelike $n$-marked graphs of genus $g$. To avoid set-theoretic issues (the category of such graphs is essentially small but not small), we fix once and for all a finite set $\mathcal{G}_{g,n}^{\text{TL}}$ of treelike $n$-marked graphs of genus $g$, one for each isomorphism class. (Note that the strict contraction of a graph in $\mathcal{G}_{g,n}^{\text{TL}}$ may not itself belong to $\mathcal{G}_{g,n}^{\text{TL}}$, and this is why we have introduced a more general notion of contraction).

In the following, whenever we write a product over a collection of graphs, the set of indices will always be assumed to be contained in $\mathcal{G}_{g,n}^{\text{TL}}$.

\label{graphs}
\subsection{Moduli of curves}

Throughout this paper, we fix natural numbers  $g$ and $n$ satisfying $2g-2+n >0$. In Sections \ref{Section: Wallcrossing} and \ref{PullbackTheta} we make the further assumption that $n \geq 1$. We denote  by $\pi \colon \Cb_{g,n} \to \Mb_{g,n}$ the universal curve over the moduli stack of stable curves.

\begin{definition} \label{Definition: LoopfreeRank}
	If $(C, p_1, \dots, p_n)$ is a stable pointed curve, we define the \textbf{dual graph} $\Gamma_{C}$ to be the $n$-marked graph whose vertices are the irreducible components of $C$,  whose edges are the nodes of $C$, whose genus map is given by assigning the geometric genus to each irreducible component, and whose markings map is the assignment $p \colon \{ 1, \dots, n \} \to \operatorname{Vert}(\Gamma_{C})$ such that  $p(j)$ is the irreducible component containing $p_j$. A stable pointed curve is \textbf{treelike} if its dual graph is treelike. (A stable curve is treelike if and only if each of its nodes is either separating, or it lies on a unique irreducible component).
	
	Given a stable marked graph $\Gamma$, we define $\Mm{g}{n}(\Gamma)$ to be the locally closed substack of $\Mb_{g,n}$ parameterizing  curves with dual graph $\Gamma$.  We define $\Mm{g}{n}^{\text{TL}} \subseteq \Mb_{g,n}$ to be the open substack parameterizing treelike curves. The locus of non-treelike curves has codimension $2$.
\end{definition}

\begin{definition} \label{Def: TwoComponentGraph}
	For a given pair $(i, S)$, we define $\Gamma(i, S)$ to be the graph with two vertices $v_1$ and $v_2$ and one edge connecting them, with genera  $g(v_1)=i$ and $g(v_2)=g-i$, and markings
	\[
		p(j) = \begin{cases}
					v_1 & \text{ if $j \in S$;} \\
					v_2 & \text{ otherwise.}
				\end{cases}
	\]
	If $\Gamma(i,S)$ is stable, the {\bf boundary divisor} $\Delta_{i,S} = \Delta_{g-i, [n] \setminus S}$ is the closure of $\Mm{g}{n}(\Gamma(i, S))$ in $\Mb_{g,n}$. The boundary divisor $\Delta_{irr}$ is the closure of the locus of irreducible singular curves. The corresponding Chow classes are denoted by $\delta_{i,S}$ and $\delta_{irr}$ respectively.
\end{definition}

 In this paper, we will often need to sum over the set of all boundary divisors in $\Mb_{g,n}$ whose preimage in $\Cb_{g,n}$ consists of two irreducible components. We introduce the following notation for the set of indices of such sums.
\begin{notation}
The restriction of the universal curve  to $\Delta_{i, S}$ has two irreducible components, unless when $n=0$ and $i=g/2$ when it is irreducible. If $n \geq 1$, we write $\C^{+}_{i, S}$ for the irreducible component that contains the markings $S$ and $\C^{-}_{i, S}$ for the other component. If $n=0$ and $i \neq g/2$, we write $\C^{+}_{i, S}$ for the irreducible component of smallest genus.

\label{Definition: AdmissiblePair}
We require all pairs $(i, S)$ with $i \in \{0, \dots, g\}$ and $S \subset \{ 1, \dots, n \}$ to satisfy the following.
\begin{itemize}
\item When $n \geq 1$, if $i=0$  then  $\#S \ge 2$, if $i=g$ then $\#S  \le 2$, and $1 \in S$.
\item When $n=0$ we assume $0<i< g/2$.
\end{itemize}
When summing over pairs $(i,S)$ we always implicitly assume that the summation ranges over pairs $(i,S)$ that satisfy this requirement.
\end{notation}


\section {Stability conditions} \label{Section: polytopeDecomposition}

In this section we define families  $\Jb_{g,n}(\phi) \to \Mm{g}{n}^{\text{TL}}$ that extend the universal degree $g-1$ Jacobian $\J_{g,n} \to \Mm{g}{n}$ and effective Cartier divisors  $\ThDiv(\phi) \subset \Jb_{g,n}(\phi)$ that extend the family of theta divisors $\Theta \subset \J_{g,n}$.  The families and the effective Cartier divisors are defined as the restriction to the treelike locus of analogous objects over $\Mb_{g,n}$. The families are indexed by  stability parameters $\phi \in V_{g,n}^{\text{TL}}$  lying in an affine space $V_{g,n}^{\text{TL}}$.  We describe the dependence of $\Jb_{g,n}(\phi)$ on $\phi$ by constructing a polytope decomposition of $V_{g,n}^{\text{TL}}$, called the stability polytope decomposition, with the property that $\phi_1$-stability concides with $\phi_2$-stability if and only $\phi_1$ and $\phi_2$ lie in a common polytope.

Our construction of the $\Jb_{g,n}(\phi)$'s is perhaps not the first construction that one might try.  A natural first approach is to define $V_{g,n}^{\text{TL}}$ to be the relative ample cone $\operatorname{Amp}$ inside the relative N\'{e}ron--Severi space $\Pic(\C_{g,n})_{\mathbb{R}}/\pi^{*}\Pic(\Mm{g}{n}^{\text{TL}})_{\mathbb{R}}$ and then for $\phi \in \operatorname{Amp}$ to set $\Jb_{g,n}(\phi)$ equal to the moduli space of degree $g-1$ rank~$1$, torsion-free sheaves that are slope semistable with respect to $\phi$.  For our purposes, this does not lead to a satisfactory theory because, as was observed in \cite[1.7]{alexeev04}, the condition of $\phi$-slope stability for degree $g-1$ sheaves is independent  of $\phi$, so this approach produces only one family $\Jb_{g,n}(\phi)$.  Furthermore, this family is a stack with points that have positive dimensional stabilizers because there are sheaves that are strictly semistable, and the presence of positive dimensional stabilizers  complicates the intersection theory of $\Jb_{g,n}(\phi)$ (see e.g.~\cite{edidin13}).  Below we modify this (unsuccessful) approach to construct a large collection of families $\Jb_{g,n}(\phi)$ that are smooth Deligne--Mumford stacks.

This section is organized as follows.  In Section~\ref{Subsection: Definitions} we define  $\phi$-stability and related concepts, in Section~\ref{Subsection: Combo} we define the stability polytope decomposition, and then in Section~\ref{Subsection: Rep} we construct the family $\Jb_{g,n}(\phi) \to \Mm{g}{n}^{\text{TL}}$ of compactified Jacobians and the family of theta divisors $\ThDiv(\phi) \subset \Jb_{g,n}(\phi)$ associated to a nondegenerate stability parameter $\phi$.  Finally  in Section~\ref{Subsection: Remarks} we make some remarks about the definitions of $V_{g,n}^{\text{TL}}$, $\Jb_{g,n}(\phi)$ and their relations to constructions from the literature.

\subsection{The stability space} \label{Subsection: Definitions}

The stability condition we study is the following.
\begin{definition} \label{Def: polytopeDecomp}
	Given a stable marked graph $\Gamma$ of genus $g$, define $V(\Gamma) \subset \mathbb{R}^{\operatorname{Vert}(\Gamma)}$ to be the affine subspace of vectors 	$\phi$ satisfying
	\[
		 \sum_{v \in \operatorname{Vert}(\Gamma)} \phi(v) = g-1.
	\]

	If $C$ is a  stable pointed curve with dual graph $\Gamma$ and $C_{0} \subset C$ is a subcurve with dual graph $\Gamma_0 \subset \Gamma$, then we write  $\deg_{\Gamma_0}(F)$ or  $\deg_{C_{0}}(F)$   for the total degree of the maximal torsion-free quotient of $F \otimes \mathcal{O}_{C_{0}}$.  We write $C_{0} \cap C_{0}^{c}$ or $\Gamma_{0} \cap \Gamma_{0}^{c}$ for the set of edges $e \in \operatorname{Edge}(\Gamma)$ that join a vertex of $\Gamma_0$ to a vertex of its complement $\Gamma_{0}^{c}$.
	
	Given $\phi \in V(\Gamma)$ we define a degree $g-1$ rank~$1$, torsion-free sheaf $F$ on a curve $C/k$ defined over an algebraically closed field to be \textbf{$\phi$-semistable} (resp.~\textbf{$\phi$-stable}) if
	\begin{equation} \label{Eqn: DefOfStability}
		\deg_{\Gamma_{0}}(F) \ge \sum \limits_{v \in \operatorname{Vert}(\Gamma_0)} \phi(v) - \frac{\#(\Gamma_{0} \cap \Gamma_{0}^{c})}{2}  \  \text{ (resp.~$>$)}
	\end{equation}
	for all proper subgraphs  $\Gamma_0 \subset \Gamma$.  We say that $\phi \in V(\Gamma)$ is \textbf{nondegenerate} if every $\phi$-semistable sheaf is $\phi$-stable.
\end{definition}
This notion of $\phi$-stability is essentially taken from \cite{oda79}, and it is closely related to other stability conditions appearing in the literature.  We discuss the connection between this condition and other stability conditions in more depth in Section~\ref{Subsection: Remarks}.

\begin{remark}
	Nondegenerate $\phi$'s exist. For example any $\phi$ with irrational coefficients is nondegenerate.
\end{remark}

\begin{remark}
	We have defined a sheaf $F$ to be $\phi$-semistable if an explicit lower bound on $\deg_{\Gamma_{0}}(F)$ holds, but this condition is equivalent to an explicit upper bound.

Given a subcurve $C_{0} \subset C$ with dual graph $\Gamma_0 := \Gamma_{C_0}$ and a rank~$1$, torsion-free sheaf $F$ of degree $g-1$, we have $\deg_{C_{0}}(F)+\deg_{C_{0}^{c}}(F)=g-1-\delta_{\Gamma_{0}}(F)$ for $\delta_{\Gamma_{0}}(F)$ the number of nodes $p \in \Gamma_{0} \cap \Gamma_{0}^{c}$ such that the stalk of $F$ at $p$ fails to be locally free.  As a consequence, the $\phi$-semistability (resp.~$\phi$-stability) inequality can be rewritten as
\begin{equation} \label{Eqn: AltDefOfStability}
		 \deg_{\Gamma_{0}}(F) \le  \sum \limits_{v \in \operatorname{Vert}(\Gamma_0)} \phi(v)  -\delta_{\Gamma_0}(F)+ \frac{\#(\Gamma_{0} \cap \Gamma_{0}^{c})}{2}  \ \text{ (resp.~$<$).}
	\end{equation}

	If we combine Inequalities~\eqref{Eqn: DefOfStability} and \eqref{Eqn: AltDefOfStability}, we obtain a third formulation of $\phi$-semistability (resp.~$\phi$-stability): $F$ is $\phi$-semistable (resp.~$\phi$-stable) if and only if
	\begin{equation} \label{Eqn: SymDefOfStability}
		\left| \deg_{\Gamma_0}(F)- \sum \limits_{v \in \operatorname{Vert}(\Gamma_0)} \phi(v) + \frac{\delta_{\Gamma_0}(F)}{2} \right| \le  \frac{\#(\Gamma_0 \cap \Gamma_0^{c})-\delta_{\Gamma_0}(F)}{2}  \  \text{ (resp.~$<$).}
	\end{equation}
\end{remark}

The condition of $\phi$-stability is closely related to slope stability, a condition we now recall.

\begin{definition} \label{Def: SlopeSS}
	If $A$ is an ample line bundle on a curve $C/k$ defined over an algebraically closed field, then the \textbf{slope} of a coherent sheaf $F$ is defined to be $\mu_{A}(F) =a/r$, where $r$ and $a$ are the coefficients of the Hilbert polynomial $\chi(F \otimes A^{\otimes n}) = r n + a$.

A rank~$1$, torsion-free sheaf $F$ is said to be \textbf{slope semistable} (resp.~slope stable) with respect to $A$  if $\mu_{A}(G) \le \mu_{A}(F)$  (resp.~$\mu_{A}(G) < \mu_{A}(F)$) for all subsheaves $G \subset F$.
	
	If furthermore $M$ is a line bundle on $C$, then we say that a rank~$1$, torsion-free sheaf $F$ is \textbf{twisted slope semistable}  (resp.~slope stable) with respect to $(A,M)$ if $F \otimes M$ is slope semistable (resp.~slope stable) with respect to $A$.
\end{definition}

When $G \subset F$ is the kernel of the natural surjection $F \to F \otimes \mathcal{O}_{C_0}$ for a subcurve $C_0 \subset C$, an explicit computation of slopes shows that the inequality $\mu_{A}(G \otimes M) \le \mu_{A}(F \otimes M)$ can be rewritten as
\begin{equation} \label{Eqn: SlopeAsPhi}
	\deg_{\Gamma_0}(F) \ge \sum \limits_{v \in \operatorname{Vert}(\Gamma_0)} \left( \frac{\deg_{v}(A)}{\deg(A)} (d+1-g+\deg(M)) + \frac{\deg_{v}(\omega_C)}{2} - \deg_{v}(M) \right) - \frac{\#(\Gamma_0 \cap \Gamma_0^{c})}{2}
\end{equation}
Here $\Gamma_0$ is the dual graph of $C_0$ and $d$ is the total degree of $F$.

This slope computation is sketched in \cite[pages~1245--1246]{alexeev04}.  Furthermore, at the bottom of \cite[pages~1244]{alexeev04}, Alexeev observes that it is sufficient to check inequality $\mu_{A}(G \otimes M) < \mu_{A}(F \otimes M)$ (resp.~$\mu_{A}(G \otimes M) \le \mu_{A}(F \otimes M)$) when $G$  is of the form $G = \ker( F \to F \otimes \mathcal{O}_{C_0})$ for some subcurve $C_0 \subset C$.  From this, we deduce the following lemma.

\begin{lemma} \label{Lemma: PhiContainsSimpson}
	Let $(C, p_1, \dots, p_n)$ be a stable pointed curve and $A$ and $M$ line bundles on $C$ with $A$ ample.  If $\phi(A, M) \in V(\Gamma_{C})$ is defined by setting for all $v \in \operatorname{Vert}(\Gamma_{C})$
	\begin{equation} \label{Eqn: SlopeToPhi}
		\phi(A,M)(v)  := \frac{\deg_{v}(A)}{\deg(A)} \deg(M) + \frac{\deg_{v}(\omega_{C})}{2}-\deg_{v}(M),
	\end{equation}
	then a degree $g-1$ rank~$1$, torsion-free sheaf $F$ is $\phi(A,M)$-semistable (resp.~$\phi(A,M)$-stable) if and only if $F$ is twisted slope semistable (resp.~stable) with respect to $(A, M)$.
\end{lemma}

Motivated by the lemma, we make the following definition.
\begin{definition}
	We define the \textbf{canonical parameter} $\phi_{\text{can}} \in V(\Gamma)$ of a stable pointed curve $C$ with dual graph $\Gamma$ by setting $\phi_{\text{can}}(v) = \frac{\deg_{v}\omega_{C}}{2}$ for all $v \in \operatorname{Vert}(\Gamma)$.
\end{definition}
Concretely $\phi_{\text{can}}(v) =  g(v)-1+\#\operatorname{Edge}(N_{v})/2$ where $N_{v} \subset \Gamma$ is the neighbourhood of $v$.   A sheaf $F$ is $\phi_{\text{can}}$-semistable if and only if it is slope semistable with respect to an ample line bundle, i.e.~$\phi_{\text{can}}=\phi(A, \mathcal{O}_{C})$ for some (equivalently all) ample $A$.

Next we define the stability space that controls families $\Jb_{g,n}(\phi)$ over $\Mm{g}{n}^{\text{TL}}$.  Recall that in the last paragraph of Section \ref{graphs} we have fixed a finite set $\mathcal{G}_{g,n}^{\text{TL}}$ of treelike graphs, one for each isomorphism class.
\begin{definition} \label{Def: StabilityInFamilies}

	Suppose that $c \colon \Gamma_1 \to \Gamma_2$ is a contraction of stable marked  graphs.  We say that $\phi_1 \in V(\Gamma_1)$ is \textbf{compatible} with $\phi_2 \in V(\Gamma_2)$ with respect to $c$  if
	\begin{equation} \label{Eqn: CompatibilityRelation}
		\phi_2(v_2) = \sum \limits_{c(v_1)=v_2} \phi_1(v_1)
	\end{equation}
	for all vertices $v_2 \in \operatorname{Vert}(\Gamma_2)$.
	
 The \textbf{stability space} $V_{g,n}^{\text{TL}} $ is defined to be the subset
	\[
		V_{g,n}^{\text{TL}} \subset \prod_{\Gamma \in \mathcal{G}_{g,n}^{\text{TL}}} V(\Gamma)
	\]
	consisting of $\phi \in  \prod_{\Gamma \in \mathcal{G}_{g,n}^{\text{TL}}} V(\Gamma)$ such that for all $\Gamma_1, \Gamma_2 \in \mathcal{G}_{g,n}^{\text{TL}}$ and all contractions $c \colon \Gamma_1 \to \Gamma_2$, the component $\phi(\Gamma_1)$ (of $\phi$ along $\Gamma_1$) is compatible with the component $\phi(\Gamma_2)$ with respect to $c$.
	
	Given a stability parameter $\phi \in V_{g,n}^{\text{TL}}$ we say that a  rank~$1$, torsion-free sheaf $F$ of degree $g-1$ on a stable pointed  curve $(C, p_1, \dots, p_n) \in \Mm{g}{n}^{\text{TL}}$ is \textbf{$\phi$-semistable} (resp.~\textbf{$\phi$-stable})  if $F$ is $\phi(\Gamma)$-semistable (resp.~$\phi(\Gamma)$-stable) for $\Gamma$ the dual graph of $C$.  We say that $\phi \in V_{g,n}^{\text{TL}}$ is \textbf{nondegenerate} if the component $\phi(\Gamma)$ is nondegenerate for all $\Gamma \in \mathcal{G}_{g,n}^{\text{TL}}$.

The \textbf{canonical parameter} $\phi_{\text{can}} \in V_{g,n}^{\text{TL}}$ is defined to be $\phi_{\text{can}} := (\phi_{\text{can}}(\Gamma))_{\Gamma \in \mathcal{G}_{g,n}^{\text{TL}}}$.
\end{definition}

In addition to compatibility with contractions, a natural condition to impose on $\phi \in \prod_{\Gamma \in \mathcal{G}_{g,n}^{\text{TL}}}$ would be invariance under automorphism. We will see in Remark \ref{Remark: InvUnderAut} that, for treelike graphs, this further assumption automatically follows from compatibility with all contractions.

\begin{remark}
	In Definition~\ref{Def: StabilityInFamilies} we defined $V_{g,n}^{\text{TL}}$ to be the subset of $\phi$'s that are compatible with contractions in order to ensure that there is a well-behaved moduli stack $\Jb_{g,n}(\phi)$ associated to a nondegenerate stability parameter $\phi$ (the existence of $\Jb_{g,n}(\phi)$ is Corollary~\ref{Cor: JbExists} below).  Without the compatibility condition, a suitable moduli stack may not exist.  The essential point is this.  Suppose $L_{\eta}$ is a line bundle on a stable pointed curve $C_{\eta}$ that specializes to $L_s$ on $C_s$ within some $1$-parameter family.  If $C_{0, \eta}$ is an irreducible component of $C_{\eta}$, then that irreducible component specializes to a subcurve $C_{0, s}$, and the degrees are related by
	\begin{equation} \label{Eqn: CompatibleDegree}
		\deg_{C_{0, \eta}}(L_{\eta}) = \deg_{C_{0, s}}(L_{s})
	\end{equation}
	(by continuity of the Euler characteristic).
	
	Equation~\eqref{Eqn: CompatibleDegree} is exactly the condition that the degree vector ${\deg}(L)$ is compatible with the contraction $c \colon \Gamma_{C_s} \to \Gamma_{C_{\eta}}$.  Thus when defining stability conditions on line bundles, it is natural to require that the degree vectors of stable line bundles are compatible with contractions, and this holds when the line bundles are the $\phi$-stable line bundles for a stability parameter $\phi$ that is compatible with contractions.
\end{remark}

We conclude the section by proving that a nondegenerate stability parameter $\phi \in V_{g,n}^{\text{TL}}$ is determined by its components $\phi(\Gamma(i,S))$ for all indices $(i,S)$ as in  Notation \ref{Definition: AdmissiblePair}.
\begin{lemma} \label{Lemma: PhiDeterminedByTwoComponents}
	The restriction of the natural projection
	\begin{equation} \label{Eqn: Projection}
		 \prod_{\Gamma \in \mathcal{G}_{g,n}^{\text{TL}}} V(\Gamma) \to   \prod_{(i,S)} V(\Gamma(i,S))
	\end{equation}
	to $V_{g,n}^{\text{TL}}$ is a bijection.
\end{lemma}
\begin{proof}
(For clarity, we only prove this in the special case where either $n \ge 1$ or $g$ is odd, but we indicate in Remark \ref{messycase} how to extend the argument to the case where $n=0$ and $g$ is even.)

First we observe that if $\Gamma$ has loops, we can apply compatibility with contractions and uniquely determine $\phi(\Gamma)$ from $\phi(\overline{\Gamma})$, where $\overline{\Gamma}$ is the unique graph in $\mathcal{G}_{g,n}^{\text{TL}}$ obtained by subsequently contracting all loops of $\Gamma$.  Therefore we can assume that $\Gamma$ is a tree with at least two edges.

We now prove injectivity.  We start by recalling that because $\phi(\Gamma)$ belongs to $V(\Gamma)$, it satisfies the equation
\begin{equation} \label{Eqn: SecondCompatibilityEquality}
\sum\limits_{v \in \operatorname{Vert}(\Gamma)} \phi(\Gamma)(v)=g-1.
\end{equation}
If $e \in \operatorname{Edge}(\Gamma)$ is  any edge, then $\Gamma - e$ has two connected components, say $\Gamma^{+}$ and $\Gamma^{-}$.  If $c_e$ is a composition of contractions that contracts all edges of $\Gamma$ except for $e$, then $\Gamma^{+}$ and $\Gamma^{-}$ are contracted to two distinct vertices, say $v^{+}$ and $v^{-}$ respectively.  The vector $\phi$ is compatible with $c_e$ if and only if the following equality is satisfied:
	\begin{equation} \label{Eqn: FirstCompatibilityEquality}
		\sum\limits_{v \in \operatorname{Vert}(\Gamma^{+})} \phi(\Gamma)(v) = \phi(c_e(\Gamma))(v^{+}).\end{equation}
(Note that for each $e \in \operatorname{Edge}(\Gamma)$ the equation above depends on the choice of an orientation of $e$, but the two resulting equations are equivalent because of Equation \eqref{Eqn: SecondCompatibilityEquality}).
	
Varying over all $e \in \operatorname{Edge}(\Gamma)$, the equations in \eqref{Eqn: FirstCompatibilityEquality} together with Equation \eqref{Eqn: SecondCompatibilityEquality} form a system of  $\#\operatorname{Edge}({\Gamma})+1 = \#\operatorname{Vert}(\Gamma)$ inhomogeneous equations in $\# \operatorname{Vert}(\Gamma)$ variables.

We claim that the associated system of homogeneous equations is nondegenerate. If $w$ is a leaf and $e_w$ is the unique edge that connects it to the rest of $\Gamma$, the unknown variables $\phi(\Gamma)(v)$ on the left-hand side of Equation \eqref{Eqn: FirstCompatibilityEquality}  all appear with a coefficient $0$, with the exception of $\phi(\Gamma)(w)$ which appears with a coefficient $1$. The claim then follows by induction on $\#\operatorname{Edge}(\Gamma)$ (the case $\#\operatorname{Edge}(\Gamma)=1$ is immediate). It follows that the projection \eqref{Eqn: Projection} restricted to $V_{g,n}^{\text{TL}}$ is injective.
	
	We now establish surjectivity.  Given
	\[
		\psi \in  \prod_{(i,S)} V(\Gamma(i,S)),
	\]
	we define $\phi$ as follows. If $\Gamma$ is a tree with at least two edges, we let $\phi(\Gamma)$ be the unique solution to the system of Equations  \eqref{Eqn: FirstCompatibilityEquality} (varying over all $e \in \operatorname{Edge}(\Gamma)$) together with  Equation \eqref{Eqn: SecondCompatibilityEquality}.

	To prove surjectivity it is enough to prove that the element $\phi \in  \prod_{\Gamma \in \mathcal{G}_{g,n}^{\text{TL}}} V(\Gamma)$ defined in the previous paragraph is compatible with all contractions.  Indeed, given any contraction $c \colon \Gamma_1 \to \Gamma_2$, define $\phi'(\Gamma_2) \in V(\Gamma_2)$ by setting
	\[
		\phi'(\Gamma_2)(v_2) = \sum \limits_{c(v_1)=v_2} \phi(\Gamma_1)(v_2).
	\]
	Then both $\phi'(\Gamma_2)$ and $\phi(\Gamma_2)$ satisfy the system of equations  \eqref{Eqn: FirstCompatibilityEquality} and  \eqref{Eqn: SecondCompatibilityEquality}, and again because that system is nondegenerate we deduce $\phi'(\Gamma_2)=\phi(\Gamma_2)$, proving that $\phi$	is compatible with $c \colon \Gamma_1 \to \Gamma_2$.
\end{proof}

\begin{remark} \label{messycase} In the proof of Lemma \ref{Lemma: PhiDeterminedByTwoComponents} the cases $g$ even and $n=0$ are special  because of the presence of a stable graph $\Gamma(g/2, \emptyset)$, with two vertices and one edge, that admits a nontrivial involution $\alpha$.

We claim that compatibility with contractions forces the equality \[\phi(\Gamma(g/2, \emptyset))= \left(\frac{g-1}{2}, \frac{g-1}{2}\right).\] Indeed, consider the graph $\Gamma'$ with two vertices of genus $\frac{g}{2}-1$ and $\frac{g}{2}$ respectively, one loop on the first edge and one edge connecting the two vertices. Applying the two different contractions $\Gamma' \to \Gamma(g/2, \emptyset)$, we deduce that the two components of $\phi(\Gamma(g/2, \emptyset))$ must be equal, and the claim follows from the fact that their sum equals $g-1$.

To complete the proof of Lemma \ref{Lemma: PhiDeterminedByTwoComponents} one needs to take extra care of the fact that if $c_e \colon \Gamma \to \Gamma(g/2, \emptyset)$ is a composition of contractions, then there is a \emph{different} composition of contractions $\alpha \circ c_e$, and the two corresponding equations \eqref{Eqn: FirstCompatibilityEquality} are equal modulo \eqref{Eqn: SecondCompatibilityEquality}. Therefore the assertion in Lemma \ref{Lemma: PhiDeterminedByTwoComponents} that the given system of equations is nondegenerate remains valid if one considers only one of the two equations arising from $c_e$ and $\alpha \circ c_e$.
\end{remark}

\begin{remark}\label{Remark: InvUnderAut} From the proof of Lemma \ref{Lemma: PhiDeterminedByTwoComponents} we can also deduce that, for $\phi \in  \prod_{\Gamma \in \mathcal{G}_{g,n}^{\text{TL}}} V(\Gamma)$, compatibility with contractions implies invariance under automorphisms i.e.~for all graphs $\Gamma\in \mathcal{G}_{g,n}^{\text{TL}}$ and all graph automorphisms $\alpha \colon \Gamma \to \Gamma$, we have $\phi(\Gamma)(v)=\phi(\Gamma)(\alpha(v))$.

We prove this assuming $n>1$ or $g$ odd (the remaining cases should be treated with extra care as in Remark \ref{messycase}). Let $\Gamma\in \mathcal{G}_{g,n}^{\text{TL}}$ and $\alpha \colon \Gamma \to \Gamma$ be an automorphism. By subsequently contracting all loops (if any), we can assume that $\Gamma$ is a tree with at least two edges.    If $c_e$ is a composition of contractions that contracts all edges of $\Gamma$ except for $e$, then so is $c_e \circ \alpha$. Equation \eqref{Eqn: FirstCompatibilityEquality} for $c_e \circ \alpha$ corresponds to the equation for $c_e$ by the permutation of unknowns $\phi(\Gamma)(v) \to \phi(\Gamma)(\alpha(v))$ (when forming Equation \eqref{Eqn: FirstCompatibilityEquality}  for $c_e$ and for $c_e \circ \alpha$ one has to choose the corresponding orientation for $e$).  Because the system of Equations \eqref{Eqn: SecondCompatibilityEquality} and \eqref{Eqn: FirstCompatibilityEquality} is nondegenerate, its unique solution satisfies $\phi(\Gamma)(v) = \phi(\Gamma)(\alpha(v))$.
\end{remark}

\subsection {The stability polytope decomposition}  \label{Subsection: Combo}
Here we define the polytope  decompositions of $V(\Gamma)$ and of $V_{g,n}^{\text{TL}}$ that describe how $\phi$-stability depends on $\phi$.

We will use the following definition and lemma to construct the decomposition.

\begin{definition}
	A subgraph $\Gamma_0 \subset \Gamma$ is said to be \textbf{elementary} if both $\Gamma_0$ and its complement $\Gamma_0^c$ are connected.
\end{definition}

\begin{remark}
	The vertex set of an elementary subgraph is  an elementary cut in the sense of \cite[page~31]{oda79}.
\end{remark}

\begin{remark}
	When $\Gamma \in \mathcal{G}_{g,n}^{\text{TL}}$ (the case of present interest), a subgraph $\Gamma_0 \subset \Gamma$ is elementary if and only if $\Gamma_0 \cap \Gamma_0^c$ consists of one edge.
\end{remark}

\begin{lemma} \label{Lemma: ElementaryFindsUnstable}
	Let $(C, p_1, \dots, p_n)$ be a stable pointed curve and $\phi \in V(\Gamma_{C})$.  A rank~$1$, torsion-free sheaf $F$ of degree $g-1$ is $\phi$-semistable (resp.~$\phi$-stable) if and only if Inequality~\eqref{Eqn: DefOfStability} holds for all elementary subgraphs of $\Gamma_{C}$.
\end{lemma}
\begin{proof}
	Set $\Gamma := \Gamma_{C}$.  It is enough to show that if $F$ satisfies  Inequality \eqref{Eqn: SymDefOfStability} for all elementary subgraphs $\Gamma_0 \subset \Gamma$,  then it satisfies the inequality for \emph{all} subgraphs.  First, consider the case where $\Gamma_{0}^{c}$ is connected.  Let $\Gamma_1, \dots, \Gamma_n$ be the connected components of $\Gamma_0$.
	
	Each of the connected components $\Gamma_1, \dots, \Gamma_n$ is an elementary subgraph of $\Gamma$.  Indeed, the complement of $\Gamma_i$ in $\Gamma$ is
	\[
		\Gamma_{i}^{c} = \Gamma_{0}^{c} \cup \Gamma_{1} \cup \Gamma_{2} \dots \cup \Gamma_{i-1} \cup \Gamma_{i+1} \cup \dots \cup \Gamma_{n},
	\]
	and we can connect each  $\Gamma_{j}$ for $j \ne i$ to $\Gamma_{0}^{c}$ as follows.  Since $\Gamma$ is connected, for $j \ne i$, we can connect any vertex in $\Gamma_{j}$ to any vertex in $\Gamma_0^{c}$ by a path in $\Gamma$.  Pick one such path $v_0, v_1, \dots, v_n$ that has minimal length.  There is no consecutive pair $v_{i}, v_{i+1}$ of vertices with  $v_i \in \operatorname{Vert}(\Gamma_{k_{1}})$ and $v_{i+1} \in \operatorname{Vert}(\Gamma_{k_2})$ for distinct $k_1, k_2$ because no edge joins $\Gamma_{k_1}$ to $\Gamma_{k_2}$.  Furthermore, the first vertex that lies in  $\Gamma_{0}^{c}$ must be $v_n$ by minimality, so the vertices $v_1, \dots, v_{n-1}$ must all lie in $\Gamma_j$.  This proves that $\Gamma_i$ is elementary.
	
	By hypothesis, Inequality~\eqref{Eqn: SymDefOfStability} holds for the subgraphs $\Gamma_1, \dots, \Gamma_n$ and combining those inequalities with the triangle inequality, we  get Inequality~\eqref{Eqn: SymDefOfStability}  for $\Gamma_0$.  This proves the lemma under the assumption that $\Gamma_0^{c}$ is connected.
	
	For arbitrary $\Gamma$ we argue as follows.  Inequality \eqref{Eqn: SymDefOfStability} is symmetric with respect to replacing $\Gamma_0$ with $\Gamma_0^{c}$, so the result follows immediately when $\Gamma_0$ is connected. When $\Gamma_0$ is not connected, the result follows by expressing $\Gamma_0$ as a union of connected components and applying the triangle inequality.
\end{proof}

\begin{definition} \label{Definition: PolytopesOneCurve}
	Let $\Gamma$ be a stable marked graph of genus $g$.  To a subgraph $\Gamma_0 \subset \Gamma$ and an integer $d \in \mathbb{Z}$ we associate the affine linear function  $\ell(\Gamma_0, d) \colon V(\Gamma) \to \mathbb{R}$ defined by
	\[
		\ell(\Gamma_0, d)(\phi)  := d- \sum_{v \in \operatorname{Vert}(\Gamma_0)}\phi(v) + \frac{\#(\Gamma_0 \cap \Gamma_{0}^{c})}{2}.
	\]

If $\Gamma_0 \subset \Gamma$ is an elementary subgraph, we call
	\[
		 H(\Gamma_0, d) := \{ \phi \in V(\Gamma) \colon  \ell(\Gamma_0, d)(\phi) =0 \}
	\]
	a \textbf{stability hyperplane}. A connected component  of the complement of all stability hyperplanes in $V(\Gamma)$
	\[
		V(\Gamma) - \bigcup\limits_{\substack{\Gamma_0 \subset \Gamma \text{ elementary}\\ d \in \mathbb{Z}}}  H(\Gamma_0, d)
	\]
	is defined to be a \textbf{stability polytope}, and the set of all stability polytopes is defined to be the \textbf{stability polytope decomposition} of $V(\Gamma)$.
	
	Given a stability parameter $\phi \in V(\Gamma)$, we denote the unique stability polytope containing $\phi$ by $\mathcal{P}(\phi)$.
\end{definition}
By definition, if $\phi_0$ is a nondegenerate stability parameter,  then
\begin{equation} \label{Eqn: StabilityPolytope}
	\mathcal{P}(\phi_0) = \{ \phi \in V(\Gamma) \colon \ell(\Gamma_0, d)(\phi)>0 \text{ for all } \ell(\Gamma_0, d) \text{ s.t.~}\ell(\Gamma_0, d)(\phi_0) > 0 \}.
\end{equation}

The stability polytope  $\mathcal{P}(\phi_0)$ is a rational bounded convex polytope because in Equation~\eqref{Eqn: StabilityPolytope}  only finitely many $\ell(\Gamma_0, d)$'s are needed to define $\mathcal{P}(\phi_0)$.

\begin{example}
	When  $\Gamma= \Gamma_C$ has one vertex (i.e.~when $C$ is irreducible), $V(\Gamma)$ is a $0$-dimensional affine space.  There are no elementary subgraphs of $\Gamma$, so there is only one stability polytope: $V(\Gamma)$ itself.
\end{example}

\begin{example} \label{Example: TwoComponentExample}
	Suppose that $\Gamma$ is the graph depicted in Figure~\ref{Figure: TwoVertexGraph}.
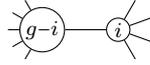
\begin{figure} \begin{tikzpicture}[baseline]
      \path(0,0) ellipse (2 and 1);
      \tikzstyle{level 1}=[counterclockwise from=-60,level distance=9mm,sibling angle=40]
      \node (A0) at (0:1) {$\scriptstyle{i}$} child child child child;
      \tikzstyle{level 1}=[counterclockwise from=120,level distance=9mm,sibling angle=30]
      \node (A1) at (180:1) {$\scriptstyle{g-i}$} child child child child child;

      \path (A0) edge [bend left=0.000000] (A1);
    \end{tikzpicture}
    \caption{A stable graph $\Gamma(i,S)$ with two vertices and one edge.}
    \label{Figure: TwoVertexGraph}\end{figure}
The associated stability polytopes are depicted in Figure~\ref{Figure: StabilityPolytopes}.  The affine space $V(\Gamma)$ is $1$-dimensional, and a stability polytope  is an open line segment with endpoints at two consecutive half-integer points.  More precisely, if $\vec{d}=(d(v_1), d(v_2)) \in V_{\mathbb{Z}}(\Gamma)$ is an integral vector, then the set of solutions to
	\begin{align*}
		d(v_1) - \phi(v_1)+1/2 &>0, \\
		d(v_2) - \phi(v_2)+1/2&>0
	\end{align*}
	is a stability polytope $\mathcal{P}(d)$ that can be described as the relative interior of the convex hull of $(d(v_1)-1/2, d(v_2)+1/2)$ and $(d(v_1)+1/2, d(v_2)-1/2)$, and every stability polytope  can be written as $\mathcal{P}(\vec{d})$ for a unique $\vec{d}$.
	
\begin{figure}[h]
    \centering
    \includegraphics[scale=0.4]{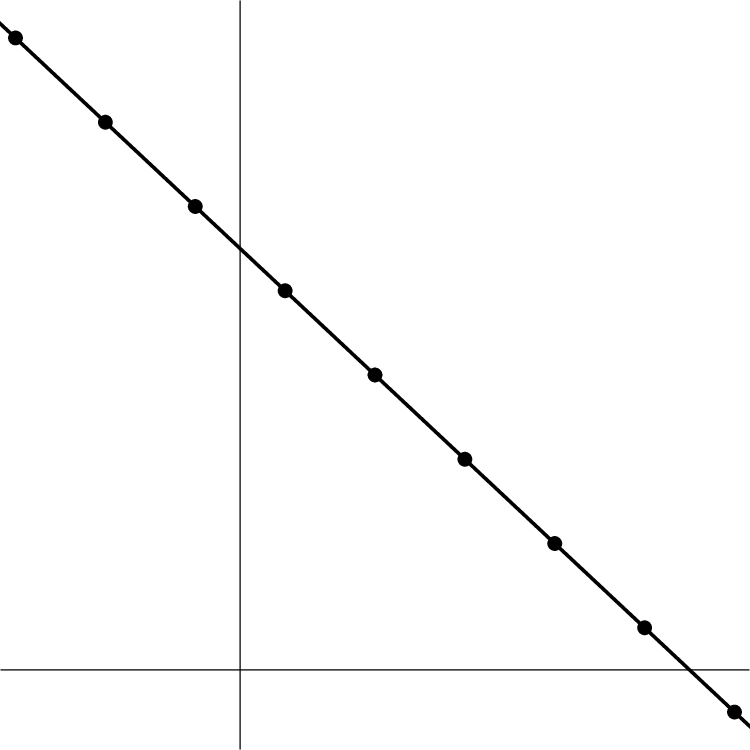}
    \caption{The stability polytopes of a two-vertex graph $\Gamma$}
    \label{Figure: StabilityPolytopes}
\end{figure}
\end{example}

One property of the curve whose dual graph is depicted in Figure~\ref{Figure: TwoVertexGraph} is that, for every  nondegenerate stability parameter $\phi$, there is a unique $\phi$-stable multidegree.  This is  true more generally  for treelike curves.  This fact is a consequence of the more general result, proven in \cite[Theorem~7.7]{oda79} for degree $d=0$ and in \cite[Theorem~4.1]{melo12} for arbitrary $d$,  that for an arbitrary nodal curve $C$ and a nondegenerate stability parameter $\phi$, the set of $\phi$-stable multidegrees of line bundles has cardinality equal to the complexity of the dual graph.   Those results imply the following lemma, but we give a self-contained proof.

\begin{lemma}	\label{Lemma: ComponentGroup}
Let $(C, p_1, \dots, p_n)$ be a stable curve, $\Gamma$ a spanning tree of the dual graph $\Gamma_{C}$, and  $\phi \in V(\Gamma_C)$ a nondegenerate stability parameter.  Then there exists a $\phi$-stable sheaf on $C$ that fails to be locally free at the nodes not in $\Gamma$. Furthermore, the multidegree of this sheaf is unique.
\end{lemma}
\begin{proof}	
	Endow  $\Gamma$ with the structure of a rooted tree by arbitrarily picking a vertex $r_0 \in \operatorname{Vert}(\Gamma)$ as the root.  We prove the lemma by working one vertex at a time, starting with the leaves and ending with the root.  Suppose $v_0 \in \operatorname{Vert}(\Gamma)$ is a vertex.  Define $\Gamma_0$ to be the complete subgraph of the dual graph $\Gamma_{C}$ with vertex set consisting of  $v_0$ and its descendants (in the rooted spanning tree).  For a sheaf $F$ that fails to be locally free at every edge not in $\Gamma$, the $\phi$-stability inequality \eqref{Eqn: SymDefOfStability} for $\Gamma_{0}$ takes the form
		\begin{equation} \label{Eqn: StabilityInequalityTree}
		\left| \deg_{\Gamma_0}(F)-\sum_{v \in \operatorname{Vert}(\Gamma_0)} \phi(v) + \frac{\delta_{\Gamma_0}(F)}{2}\right| < \frac{1}{2}.
	\end{equation}
	This proves that the partial degree $\deg_{\Gamma_0}(F)$ of a $\phi$-stable sheaf is uniquely determined (and equal to the integer nearest to $\sum_{v \in \operatorname{Vert}(\Gamma_0)} \phi(v)+\delta_{\Gamma_0}(F)/2$).  Reverse induction on the depth of $v_0$ shows that $\deg_{v_0}(F)$ is also uniquely determined.  Indeed, the inductive hypothesis then implies $\deg_{\Gamma_0 \setminus \{ v_0 \}}(F)$ is uniquely determined and
	\begin{equation} \label{Eqn: InductionFormula}
		\deg_{\Gamma_0}(F) = \deg_{v_0}(F)+ \deg_{\Gamma_0 \setminus \{ v_0\}}(F)+\#\{\text{edges from $v_0$ to $\Gamma_0 \setminus \{ v_0 \}$ not in $\Gamma_0$}\}.
	\end{equation}
	
	For existence, we observe that any rank~$1$, torsion-free sheaf $F$ that fails to be locally free at the nodes not in $\Gamma$ and satisfies \eqref{Eqn: InductionFormula} is $\phi$-stable.  (Such sheaves exist.  Take, for example,  $F$ to be the direct image of a line bundle of suitable multidegree under the partial normalization of $C$ given by resolving the nodes not in $\Gamma$.)
	
	By the construction of $F$, the $\phi$-stability inequality  \eqref{Eqn: StabilityInequalityTree} is satisfied when $\Gamma_0$ is the subgraph of $\Gamma_C$ that contains $v_0$ and its descendants in $\Gamma$, and the same is true of the complement of $\Gamma_0$ by symmetry.   We observe that these subgraphs are the subgraphs of $\Gamma_C$ obtained as  complements of elementary subgraphs of $\Gamma$.

	We  deduce that the stability inequality holds for all subgraphs of $\Gamma_C$  as follows.  Given an arbitrary subgraph $\Gamma_0$ of $\Gamma_C$, we decompose its vertex set as \[\operatorname{Vert}(\Gamma_0) = \operatorname{Vert}(\Gamma_1) \sqcup \dots \sqcup \operatorname{Vert}(\Gamma_m)\] for subgraphs $\Gamma_1, \dots, \Gamma_m$ of $\Gamma_C$ of the type discussed in the above paragraph (vertex set equal to the vertex set of an elementary subgraph of $\Gamma$), and with the property that no edge of $\Gamma$ connects $\Gamma_i$ to $\Gamma_j$ for $i \ne j$ for $i=1, \ldots, m$.  (To decompose, consider the connected components of the subgraph of $\Gamma$ with vertex set  $\operatorname{Vert}(\Gamma_0)$.)
	
	The stability inequality for $F$ on $\Gamma_1, \dots, \Gamma_m$ implies the stability inequality for $F$ on $\Gamma_0$.  Indeed, since $F$ fails to be locally free at every edge connecting $\Gamma_i$ to $\Gamma_j$ for $i \ne j$, we have
	\[
		\#(\Gamma_0 \cap \Gamma_0^{c})- \delta_{\Gamma_0}(F) = \#(\Gamma_1 \cap \Gamma_1^{c})- \delta_{\Gamma_1}(F)+ \dots +\#(\Gamma_m \cap \Gamma_m^{c})- \delta_{\Gamma_m}(F)=m.
	\]
	We deduce the stability inequality for $F$ on $\Gamma_0$ from the inequalities for $F$ on $\Gamma_1, \dots, \Gamma_m$ using the triangle inequality:
\begin{align*}
\left| \deg_{\Gamma_0}(F) - \sum_{v \in \operatorname{Vert}(\Gamma_0)} \phi(v) + \frac{\delta_{\Gamma_0}(F)}{2} \right|
	=	& \left| \sum_{i=1}^m \left( \deg_{\Gamma_i}(F) - \sum_{v \in \operatorname{Vert}(\Gamma_i)} \phi(v) + \frac{\delta_{\Gamma_i}(F)}{2} \right) \right| \\
	 \leq &     \sum_{i=1}^m \left| \deg_{\Gamma_i}(F) - \sum_{v \in \operatorname{Vert}(\Gamma_i)} \phi(v) + \frac{\delta_{\Gamma_i}(F)}{2} \right| \\
	 <&  \sum_{i=1}^m \frac{ \#(\Gamma_i \cap \Gamma_i^{c})- \delta_{\Gamma_i}(F)}{2} \\
	 =&  \#(\Gamma_0 \cap \Gamma_0^{c})- \delta_{\Gamma_0}(F).
\end{align*}
This concludes our proof of Lemma~\ref{Lemma: ComponentGroup}.
\end{proof}

We now apply Lemma~\ref{Lemma: ComponentGroup} to prove that the stability polytopes from Definition~\ref{Definition: PolytopesOneCurve} witness the variation of $\phi$-stability.

\begin{lemma} \label{Lemma: StabilityPolytope}
	Let $\phi_1, \phi_2 \in V(C)$ with $\phi_1$ nondegenerate.  Then  every $\phi_1$-stable sheaf is $\phi_2$-semistable if and only if $\phi_2 \in \overline{\mathcal{P}}(\phi_{1})$ (the closure of the unique polytope containing $\phi_1$; see Definition \ref{Definition: PolytopesOneCurve}).
\end{lemma}
\begin{proof}
If $\phi_2 \in \overline{\mathcal{P}}(\phi_{1})$, then by definition
\[
	\ell(\Gamma_0, d)(\phi_2) \ge 0 \text{ for all } \ell(\Gamma_0, d) \text{ s.t.~}\ell(\Gamma_0, d)(\phi_1) > 0.
\]
In other words, if $F$ is a $\phi_1$-stable sheaf, then $F$ satisfies the $\phi_2$-semistability inequality for all elementary subgraphs $\Gamma_0$ of the dual graph $\Gamma_{C}$  and hence is $\phi_2$-semistable by Lemma~\ref{Lemma: ElementaryFindsUnstable}.

Conversely, suppose that $\phi_2$ does not lie in  $\overline{\mathcal{P}}(\phi_1)$ so that  there exist $d \in \mathbb{Z}$ and $\Gamma_0 \subset \Gamma_{C}$ an elementary subgraph such that $\ell(\Gamma_0, d)(\phi_1)>0$ but $\ell(\Gamma_0, d)(\phi_2)<0$.  Construct a spanning tree $\Gamma$ for $\Gamma_{C}$ by picking spanning trees for $\Gamma_0$ and $\Gamma_0^{c}$ and then taking $\Gamma$ to be subgraph  obtained by connecting the spanning trees by an arbitrarily chosen edge from $\Gamma_0$ to $\Gamma_0^c$.  We will show that a $\phi_1$-stable sheaf $F$ that fails to be locally free at every node not in $\Gamma$ is not $\phi_2$-stable.  (Such a sheaf exists by Lemma~\ref{Lemma: ComponentGroup}).

By construction, $1=\#(\Gamma_0 \cap \Gamma_0^{c}) - \delta_{\Gamma_0}(F)$, so the $\phi_1$-stability inequality~\eqref{Eqn: AltDefOfStability} takes the form
\begin{displaymath}
	- \deg_{\Gamma_0}(F)+ \sum \limits_{v \in \operatorname{Vert}(\Gamma_0)} \phi_{1}(v) - \frac{\delta_{\Gamma_0}(F)}{2} > -1/2.
\end{displaymath}
We deduce that
\begin{align*}
	d - \deg_{\Gamma_0}(F) =& \sum_{v \in \operatorname{Vert}(\Gamma_0)} \phi_{1}(v)+\frac{-\delta_{\Gamma_0}(F)}{2}-\deg_{\Gamma_0}(F)+ \\
						& + d + \frac{\#(\Gamma_0 \cap \Gamma_0^{c})}{2}- \sum_{v \in \operatorname{Vert}(\Gamma_0)} \phi_{1}(v) +\\
						& + \frac{\delta_{\Gamma_0}(F) - \#(\Gamma_{0} \cap \Gamma_{0}^{c})}{2} \\
						& > -1/2+0+ -1/2 \\
						& =-1.
\end{align*}
We conclude  that $d-\deg_{\Gamma_0}(F) \ge 0$  since $d-\deg_{\Gamma_0}(F)$ is an integer.  It follows that \[\deg_{\Gamma_0}(F) \le d < \sum_{v \in \operatorname{Vert}(\Gamma_0)} \phi_2(v) - \frac{\#(\Gamma_0 \cap \Gamma_0^{c})}{2},\] contradicting the $\phi_2$-stability inequality~\eqref{Eqn: DefOfStability} and therefore proving that $F$ is not $\phi_2$-stable.
\end{proof}

\begin{remark}
	Lemma~\ref{Lemma: StabilityPolytope} becomes false if the stability hyperplanes are defined to be the subsets $\{ \ell(\Gamma_0,d)(\phi) =0 \}$ with $\Gamma_0 \subset \Gamma$ a possibly non-elementary subgraph.  For example, if  $\Gamma$ is the graph depicted in Figure~\ref{Figure: ThreeTree}, then the decomposition by solid rectangles in Figure~\ref{Figure: StabilityDecomp} is the stability polytope decomposition of $V(\Gamma)$ (or more precisely its isomorphic image under the projection $V(\Gamma) \to \mathbb{R}^{2}$, $\phi \mapsto (\phi(v_1), \phi(v_2))$).  The subdivision of the polytope decomposition given by the dotted and solid lines is the decomposition by the hyperplanes $\{ \ell(\Gamma_0, d)(\phi)=0 \}$ with $\Gamma_0 \subset \Gamma$ a possibly non-elementary subgraph.
	
	Suppose that  $\Gamma=\Gamma_C$ is the dual graph of $C$.  When $\phi \in V(\Gamma)$ crosses a dotted line, the set of vectors $\vec{d} \in \mathbb{Z}^{\operatorname{Vert}(\Gamma)}$ satisfying \[\left| \vec{d}(\Gamma_0)- \sum \limits_{v \in \operatorname{Vert}(\Gamma_0)} \phi(v_0) \right| \le \frac{\#(\Gamma_{0} \cap \Gamma_0^{c})}{2}\] changes, but the subset of vectors of the form $\vec{d}={\deg}(F)(C)$ for $F$ a $\phi$-semistable sheaf does not change.
	
	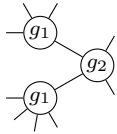
\begin{figure}[h]
    \centering
    \begin{tikzpicture}[baseline]
      \path(0,0) ellipse (2 and 2);
      \tikzstyle{level 1}=[counterclockwise from=-60,level distance=9mm,sibling angle=120]
      \node (A0) at (0:1) {$\scriptstyle{g_2}$} child child;
      \tikzstyle{level 1}=[counterclockwise from=60,level distance=9mm,sibling angle=60]
      \node (A1) at (120:1) {$\scriptstyle{g_1}$} child child child;
      \tikzstyle{level 1}=[counterclockwise from=180,level distance=9mm,sibling angle=40]
      \node (A2) at (240:1) {$\scriptstyle{g_1}$} child child child child;

      \path (A0) edge [bend left=0.000000] (A1);
      \path (A0) edge [bend left=0.000000] (A2);
    \end{tikzpicture}
    \caption{A tree $\Gamma$ with three vertices, $g_1+g_2+g_3=g$}
    \label{Figure: ThreeTree}
\end{figure}

\begin{figure}[h]
    \centering
    \includegraphics[scale=0.7]{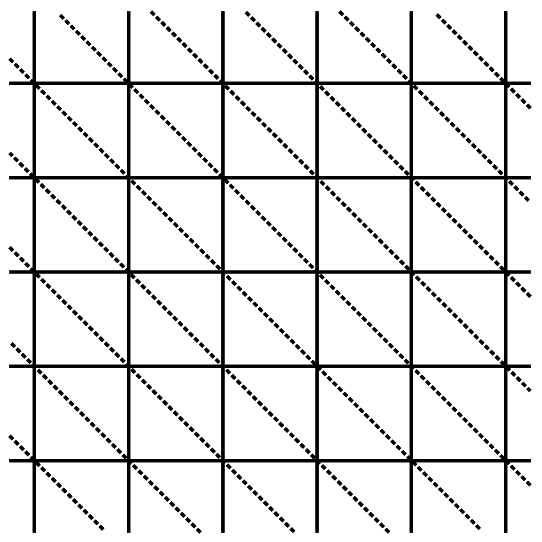}
    \caption{The stability polytopes of $V(\Gamma)$}
    \label{Figure: StabilityDecomp}
\end{figure}
\end{remark}

We now define the combinatorial objects that control the stability conditions over $\Mm{g}{n}^{\text{TL}}$.
\begin{definition} \label{Definition: PolytopesFamilyCurves}
	For $\Gamma \in \mathcal{G}_{g,n}^{\text{TL}}$ a treelike graph,  $\Gamma_0 \subset \Gamma$ an elementary subgraph, and $d \in \mathbb{Z}$ an integer, we define
	\[
		H(\Gamma, \Gamma_0, d):= \{ \phi \in V_{g,n}^{\text{TL}} \colon \ell(\Gamma_0, d)(\phi(\Gamma)) =0 \}
	\]
	to be a \textbf{stability hyperplane}.
	
	A \textbf{stability polytope} for $V_{g,n}^{\text{TL}}$ is defined to be a connected component of
	\[
		V_{g,n}^{\text{TL}} - \bigcup\limits_{\substack{ \Gamma_0 \subset \Gamma \text{ elementary}\\ d \in \mathbb{Z}}} H(\Gamma, \Gamma_0, d).
	\]
	The collection  of all stability polytopes is defined to be the \textbf{stability polytope decomposition} of $V_{g,n}^{\text{TL}}$.

	Given a stability parameter $\phi \in V_{g,n}^{\text{TL}}$, we denote the unique stability polytope containing $\phi$ by $\mathcal{P}(\phi)$.
\end{definition}
As with the stability polytopes of $V(\Gamma)$, the stability polytopes of $V_{g,n}^{\text{TL}}$ are also rational bounded convex polytopes.

Having shown in Lemma~\ref{Lemma: PhiDeterminedByTwoComponents} that a stability parameter $\phi \in V_{g,n}^{\text{TL}}$ is determined by its restriction to all $2$-vertex treelike graphs $\Gamma(i,S)$, we now prove analogous statements about stability hyperplanes and polytopes.  To shorten notation, write $H(i,S,d)$ for the hyperplane  $H(\Gamma(i,S), \Gamma_0,d)$ associated to the $2$-vertex graph $\Gamma(i, S)$ and the subgraph $\Gamma_0 \subset \Gamma$ consisting of the vertex $v_1$ (defined as in Definition~\ref{Def: TwoComponentGraph}).

\begin{lemma} \label{Lemma: HyperplanesByTwoComponents}
Every stability hyperplane $H \subset V_{g,n}^{\text{TL}}$ can be written  as the hyperplane associated to a  treelike graph with two vertices.  In other words,
	\begin{equation} \label{unambiguous}
		H = H(i,S,d)
	\end{equation}
	for  some $(i,S)$.
\end{lemma}
\begin{proof}
Let $\Gamma \in \mathcal{G}_{g,n}^{\text{TL}}$ be an arbitrary treelike graph and $\Gamma_0 \subset \Gamma$ an elementary subgraph. Consider a composition of contractions from $\Gamma$ that contracts $\Gamma_0$ to a vertex $w$ and its complement $\Gamma_0^{c}$ to a vertex $w^c$, and then contracts all resulting loops. The resulting graph $\Gamma'$ with two vertices $w$ and $w^c$ is isomorphic to $\Gamma(i,S)$ for some $(i,S)$.  By inductively applying compatibility with contractions, we find $\phi(\Gamma')(w) = \sum_{v \in \Gamma_0} \phi(\Gamma)(v)$. This implies that 	$	 H(\Gamma, \Gamma_0, d)$ equals $H(\Gamma(i,S), d)$.\end{proof}

Lemma~\ref{Lemma: HyperplanesByTwoComponents} implies that when $\phi \in V_{g,n}^{\text{TL}}$ varies in such a way that  $\phi$-semistability changes, that variation is already witnessed  over the dual graph $\Gamma(i,S)$ of a boundary divisor $\Delta_{i,S}$ as in Notation \ref{Definition: AdmissiblePair}.  As a corollary, we obtain the following description of the stability polytopes.

\begin{corollary} \label{Corollary: StabilityDeterminedByTwoComponents}
	Given a stability polytope $\mathcal{P}(\Gamma(i,S))$ for every stable marked graph $\Gamma(i,S)$, there  exists a unique stability polytope $\mathcal{P} \subset V_{g,n}^{\text{TL}}$ such that the projection under \eqref{Eqn: Projection} onto $V(\Gamma(i,S))$ is $\mathcal{P}(\Gamma(i,S))$ for all $\Gamma(i,S)$.
\end{corollary}
\begin{proof}
Uniqueness is Lemma~\ref{Lemma: PhiDeterminedByTwoComponents}. To prove existence, by the same lemma there exists $\phi_0 \in V_{g,n}^{\text{TL}}$ satisfying $\phi_{0}(\Gamma(i,S)) \in \mathcal{P}(\Gamma(i,S))$ for all stable graphs $\Gamma(i,S)$.  By Lemma~\ref{Lemma: HyperplanesByTwoComponents}, $\phi_0$ is not contained in a stability hyperplane, so it is contained in a unique stability polytope that satisfies the desired condition by the same lemma.
\end{proof}

To study the set of stability polytopes, we introduce the following group action. Let $\Pic^{0}(\C_{g,n}/\Mm{g}{n}^{\text{TL}})$ be the subgroup of the group $\Pic(\C_{g,n})/\pi^{*}(\Pic(\Mm{g}{n}^{\text{TL}}))$ generated by the images of line bundles $L$ whose restriction to any fiber of  $\pi \colon \C_{g,n} \to \Mm{g}{n}^{\text{TL}}$ has total degree $0$.  The \textbf{natural action} of $\Pic^{0}(\C_{g,n}/\Mm{g}{n}^{\text{TL}})$ on $V_{g,n}^{\text{TL}}$ is defined to be the multidegree translation
\[
\phi \mapsto \phi + \deg(L),
\]
where $\deg(L)(\Gamma_C):= \deg(L)(C)$ is the multidegree of the restriction of $L$ to $C$.

\begin{lemma}
	The natural action of $\Pic^{0}(\C_{g,n}/\Mm{g}{n}^{\text{TL}})$ on $V_{g,n}^{\text{TL}}$ maps stability polytopes to stability polytopes.
\end{lemma}
\begin{proof}	
	This follows from the identity
	\[
		\ell(\Gamma, \Gamma_0, d)(\deg(L)+\phi) = \ell(\Gamma, \Gamma_0, d- \deg_{\Gamma_0}(L))(\phi).
	\] \end{proof}

 The subschemes $\C^{\pm}_{i, S} \subset \C_{g,n}$ of the universal curve associated with each component over $\Delta_{i,S}$ (see Notation~\ref{Definition: AdmissiblePair}) are effective Cartier divisors, so their associated line bundles $\mathcal{O}( \C^{\pm}_{i, S} )$ are defined, and we use them to prove the following lemma.

\begin{lemma} \label{Lemma: Transitivity}
	The subgroup $W_{g,n}$ of $\Pic^{0}(\C_{g,n}/\Mm{g}{n}^{\text{TL}})$ generated by the line bundles $\mathcal{O}(\C^{\pm}_{i, S})$ acts freely and transitively on the set of stability polytopes in $V_{g,n}^{\text{TL}}$.
\end{lemma}
\begin{proof}
	Consider the line bundle $L := \mathcal{O}(\C^{+}_{i, S})$ associated to a pair $(i, S)$ as in Notation~\ref{Definition: AdmissiblePair}.  Its multidegree vector ${\deg}(L)$  satisfies
	\begin{equation} \label{Eqn: DescriptionOfActions}
		{\deg}(F)(\Gamma(i',S')) = \begin{cases}
										(-1, +1) 	& \text{ if $(i',S')=(i, S)$;}\\
										(0, 0)		& \text{ if $(i',S') \ne (i, S)$.}
									\end{cases}
	\end{equation}
  Using the description of stability polytopes associated to a graph $\Gamma(i,S)$ given in Example~\ref{Example: TwoComponentExample}, and  from Corollary~\ref{Corollary: StabilityDeterminedByTwoComponents}, we conclude that $W_{g,n}$ acts transitively on $V_{g,n}^{\text{TL}}$.  To see that the action is free, observe that an element of $\Pic^{0}(\C_{g,n}/\Mm{g}{n}^{\text{TL}})$ acts as translation by its multidegree, so an element acting trivially must have trivial multidegree and the only such element is the identity (because rational and numerical equivalence coincide for line bundles on $\C_{g,n}$).
\end{proof}

\begin{remark} \label{phicanvertex} By Lemma  \ref{Lemma: HyperplanesByTwoComponents} the maximum number of stability hyperplanes passing through a given point of $V_{g,n}^{\text{TL}}$ equals the number of stable graphs as in Notation \ref{Definition: AdmissiblePair}, which is also the dimension of $V_{g,n}^{\text{TL}}$, so the points for which this maximum is attained are vertices of the stability polytope decomposition.  The proof of Lemma \ref{Lemma: Transitivity} also shows that the action of $W_{g,n}$ on the set of vertices of the stability polytope decomposition is free and transitive. It follows from its definition that the canonical parameter $\phi_{\text{can}}$ is one of these vertices.
\end{remark}

\subsection{Representability}  \label{Subsection: Rep}
In this section we construct, for $\phi \in V_{g,n}^{\text{TL}}$ nondegenerate,  a family $\Jb_{g,n}(\phi) \to \Mm{g}{n}^{\text{TL}}$ of compactified Jacobians and a family of theta divisors $\ThDiv(\phi) \subset \Jb_{g,n}(\phi)$, and then describe their properties. Both $\Jb_{g,n}(\phi)$ and $\ThDiv(\phi)$ are constructed as  restrictions  of analogous objects over $\Mb_{g, n}$ that are denoted by $\Jb^{g-1}_{g,n}(A,M)$ and by $\ThDiv(A,M)$ (where $A,M$ are certain line bundles on the universal curve $\Cb_{g,n}$).   In this section we will often need to assume $n \geq 1$, an assumption that we use in the proof of  Lemmas~\ref{existsAM} and \ref{Lemma: SectionExists}.

\begin{definition} \label{Def: familytofsheaves}
	Given  a $k$-scheme $T$ and a family of  stable pointed  curves $(C/T, p_1, \dots, p_n) \in \Mb_{g,n}(T)$,  a \textbf{family of rank~$1$, torsion-free sheaves} on $C/T$ is a locally finitely presented $\mathcal{O}_{C}$-module $F$ that is $\mathcal{O}_{T}$-flat and has rank~$1$, torsion-free fibers.
\end{definition}

Our compactified universal Jacobians are constructed using Simpson's result \cite[Theorem 1.21]{simpson94}. For $\phi$ nondegenerate, the moduli space of $\phi$-stable sheaves of degree $g-1$  is not a moduli space of slope semistable sheaves because, as Lemma~\ref{Lemma: PhiContainsSimpson} shows, slope semistability coincides with $\phi$-semistability for $\phi=\phi_{\text{can}}$, which is a (maximally) degenerate parameter as we observed in Remark \ref{phicanvertex}.  We can, however, deduce representability from Simpson's result in a straightforward manner.

\begin{definition} \label{Def: moduliStacks}
	Given $A$ and $M$  line bundles on the universal curve $\Cb_{g, n}\to \Mb_{g,n}$ with $A$ ample relative to $\Mb_{g,n}$, we say that a family of rank~$1$ torsion-free sheaves is \textbf{twisted slope (semi)stable} with respect to $(A,M)$ if this condition holds fiberwise in the sense of Definition~\ref{Def: SlopeSS}.  We define $\Jb_{g,n}^{d, \text{pre}}(A, M)$ to be the category fibered in groupoids whose objects are  tuples $(C, p_1, \dots, p_n; F)$ consisting of a family of stable $n$-pointed curves $(C/T, p_1, \dots, p_n)$ of genus $g$, and a family of rank~$1$ torsion-free sheaves $F$ of degree $d$ on $C/T$ that is twisted slope (semi)stable with respect to $(A,M)$.  The morphisms of $\Jb_{g,n}^{d, \text{pre}}(A, M)$ over a $k$-morphism $t \colon T \to T'$ are pairs of an isomorphism of pointed curves $\widetilde{t} \colon (C, p_1, \dots, p_n) \cong (C'_{T}, (p'_1)_{T}, \dots, (p'_n)_{T})$, and an isomorphism of $\mathcal{O}_{C}$-modules $F \cong \widetilde{t}^{*}(F'_{T})$.
	
	For every object $(C, p_1, \dots, p_n; F)$ of $\Jb_{g,n}^{d, \text{pre}}(A, M)(T)$ the rule that sends $g \in \mathbb{G}_{m}(T)$ to the automorphism of $F$ defined by multiplication by $g$ defines an embedding $\mathbb{G}_{m}(T) \to \Aut(C, p_1, \dots, p_n; F)$  that is compatible with pullbacks.  The image of this embedding is contained in the center of the automorphism group, so the rigidification stack in the sense of \cite[Section~5.1.4]{abramovich} is defined, and we call this stack the \textbf{compactified (twisted slope (semi)stable) universal Jacobian} $\Jb^{d}_{g,n}(A, M)$.
\end{definition}

Note that the rigidification is defined in \cite[Section~5.1.4]{abramovich} for an algebraic stack, but their construction is more generally valid for stacks.

\begin{proposition} \label{existence1}
	Given line bundles $A$ and $M$ as in Definition~\ref{Def: moduliStacks} with the property that, for degree $d$ sheaves,  twisted slope semistability  coincides with twisted slope stability, then $\Jb^{d}_{g,n}(A, M)$ is a proper Deligne--Mumford stack, and the forgetful map $\Jb^{d}_{g,n}(A,M) \to \Mb_{g,n}$ is representable.
\end{proposition}
\begin{proof}
	When $M$ is the trivial line bundle $\mathcal{O}$ (so twisted slope semistability coincides with slope semistability), this is  \cite[Theorem~1.21]{simpson94}. The conclusion in loc.~cit.~that \'{e}tale locally a universal family of sheaves exists is equivalent to the representability of the forgetful morphism.  Tensoring with $M$ defines an isomorphism $\Jb_{g, n}^{d}(A, M) \cong \Jb^{d+e}(A, \mathcal{O})$ (for $e$ the degree of $M$ on a fiber), so we deduce the general case from the case $M=\mathcal{O}$.
\end{proof}

\begin{remark}
	If the hypothesis that stability coincides with semistability is dropped, then the authors expect that $\Jb^{d}_{g, n}(A, M)$ is still an algebraic stack, but  then the forgetful morphism $\Jb^{d}_{g,n}(A, M) \to \Mb_{g,n}$ is not representable, and $\Jb^{d}_{g,n}(A, M)$ is not Deligne--Mumford.  We do not pursue this issue here because we have no use for these more general families in this paper.
\end{remark}

We now show that when $n \geq 1$ there exist $A$ and $M$ satisfying the hypotheses of Proposition~\ref{existence1}.

\begin{lemma} \label{existsAM}
	Let $n \geq 1$. Then there exist line bundles $A, M$ as in Definition~\ref{Def: moduliStacks} such that, for degree $d$ sheaves, twisted slope semistability with respect to $(A,M)$ coincides with twisted  slope stability with respect to $(A,M)$.
\end{lemma}
\begin{proof}
Pick an odd number $B>2g-2+n-1$ and then set
	\begin{align*}
		b :=& B-(2g-2+n-1), \\
		A :=& \omega_{\pi}(b p_1+p_2+\dots+p_n), \text{ and }\\
		M :=& \mathcal{O}(  (g-d) \cdot p_1).
	\end{align*}
	It is enough to show that, for all subcurves $C_0 \subset C$, equality does not hold in the twisted stability inequality \eqref{Eqn: SlopeAsPhi}.

If equality held, then  $\deg_{\Gamma_0}(A)/\deg(A)$ would be a half-integer for some proper subgraph $\Gamma_0 \subset \Gamma_{C}$, but this is impossible because the slope $\deg_{\Gamma_0}(A)/\deg(A)$ is of the form $k/B$ for $k \in \mathbb{Z}$, $0<k<B$.
\end{proof}

Lemma~\ref{existence1} combined, with Lemma~\ref{existsAM} implies the existence of many Deligne--Mumford stacks $\Jb^{d}_{g, n}(A, M) \to \Mb_{g, n}$.  The following lemma describes important properties of these stacks.

\begin{lemma} \label{flatandsmooth}
	For $(A, M)$ such that twisted slope semistability coincides with twisted slope stability, the representable forgetful map $\Jb^{d}_{g,n}(A,M) \to \Mb_{g,n}$ is flat, and  $\Jb^{d}_{g,n}(A,M)$ is $k$-smooth.
\end{lemma}
\begin{proof}
To ease notation, we temporarily write $\Mb$ for $\Mb_{g,n}$ and $\Jb$ for $\Jb^{d}_{g,n}(A,M)$.   Given a closed point $x \in \Jb$ that corresponds to a pair $(C, F)$, the completed local ring $\widehat{\mathcal{O}}_{\Jb, x}$ is a versal deformation ring for the deformation problem of deforming $(C, F)$.  This ring is described as a power series ring in \cite[Corollary~3.17]{kass12b} (where the ring is denoted by $R_2$), i.e.~$\Jb$ is $k$-smooth.
	
For flatness, we argue as follows.  Set $y$ equal to the image of $x$ under $\Jb \to \Mb$.  Define   $\widehat{\mathcal{O}}_{\Mb}^{\text{loc}}$ to be the versal deformation ring for the deformation problem of deforming the product of the completed local rings of $C$ at its nodes and define  $\widehat{\mathcal{O}}_{\Jb}^{\text{loc}}$ analogously.  The natural map
	\[
		\widehat{\mathcal{O}}_{\Mb, y} \widehat{\otimes}_{\widehat{\mathcal{O}}_{\Mb}^{\text{loc}}} \widehat{\mathcal{O}}_{\Jb}^{\text{loc}} \to \widehat{\mathcal{O}}_{\Jb, x}
	\]
	is formally smooth by \cite[Proposition~A.1]{fantechi99}.  Furthermore, the natural map $\widehat{\mathcal{O}}_{\Mb}^{\text{loc}} \to \widehat{\mathcal{O}}_{\Mb, y}$ is formally smooth by \cite[Proposition~(1.5)]{deligne69}, and $\widehat{\mathcal{O}}_{\Jb}^{\text{loc}} \to \widehat{\mathcal{O}}_{\Jb, x}$ is flat by the explicit description in \cite[Lemma~3.14]{kass12b} (which identifies the map as a completed tensor product of copies of the identity map and copies of the natural map $k[[t]] \to k[[t, s_1, s_2]]/ s_1 s_2 - t$).  Since smooth maps are flat, and  flatness is preserved by pullback  and completion, we have  described $\widehat{\mathcal{O}}_{\Mb, y} \to \widehat{\mathcal{O}}_{\Jb, x}$ as a composition of flat homomorphisms, hence this last map is itself flat.
\end{proof}

We now turn our attention to the theta divisor and its associated Chow class.  When $d=g-1$, the theta divisor  $\ThDiv(A,M) \subset \Jb_{g,n}^{g-1}(A,M)$ is an effective divisor supported on the locus of sheaves that admit a nonzero global section, but it is not uniquely determined by its support because $\ThDiv(A,M)$ can be nonreduced (see Corollary~\ref{Corollary: Nonreduced}).  We define $\ThDiv(A,M)$ using the formalism of the determinant of cohomology, a formalism  we use in Section~\ref{Section: Wallcrossing} to compute intersection numbers.  More precisely, the theta divisor is defined in terms of the cohomology of the following sheaf.
\begin{definition}
	The \textbf{universal family of sheaves} $F_{\text{uni}}$ on $\Jb_{g,n}^{d, \text{pre}}(A,M) \times_{\Mb_{g,n}} \Cb_{g,n}$ is defined to be the family of sheaves that corresponds to the identity under the $2$-Yoneda Lemma.  A sheaf $F_{\text{tau}}$ on $\Jb_{g,n}^d(A,M) \times_{\Mb_{g,n}} \Cb_{g,n}$ is defined to be a \textbf{tautological family of sheaves} if $F_{\text{tau}}$ is the pullback $(\sigma \times 1)^{*}F_{\text{uni}}$ of the universal sheaf for some section $\sigma$ of the rigidification morphism $\Jb_{g,n}^{d, \text{pre}}(A,M) \to \Jb_{g,n}^d(A,M)$.
\end{definition}
Concretely, $F_{\text{tau}}$ is a $\Jb_{g,n}^d(A,M)$-flat family of rank~$1$ torsion-free sheaves on \[\Jb_{g,n}^d(A,M) \times_{\Mb_{g,n}} \Cb_{g,n}\] such that the restriction to the fiber of \[\Jb_{g,n}^d(A,M) \times_{\Mb_{g,n}} \Cb_{g,n} \to \Jb_{g,n}^d(A,M)\] over a point $(C, p_1, \dots, p_n; F) \in \Jb_{g,n}^d(A,M)$ is isomorphic to $F$. 

\begin{lemma} \label{Lemma: SectionExists}
	Let $n \geq 1$. Then the rigidification morphism $\Jb_{g,n}^{d, \text{pre}}(A,M) \to \Jb_{g,n}^d(A,M)$ admits a section.  In particular, $\Jb_{g,n}^d(A,M)$ admits a tautological family $F_{\text{tau}}$.
\end{lemma}
\begin{proof}
	A section is defined by rigidifying sheaves along the marking $p_1$.  More formally, consider the morphism $\Jb_{g,n}^{d, \text{pre}}(A,M) \to \Jb_{g,n}^{d, \text{pre}}(A,M)$ that sends a tuple $(C/T, p_1, \dots, p_n; F)$  to $(C/T, p_1, \dots, p_n; F \otimes (\pi_{T})^{*} (p_{1}^{*}(F)^{-1}))$.  Here $\pi_{T} \colon C \to T$ is the structure morphism.  To see this morphism is well-defined, observe $F \otimes \pi_{T}^{*} (p_{1}^{*}(F)^{-1}) \otimes M$ is a flat family of $A$-stable sheaves because this sheaf is Zariski locally isomorphic to $F\otimes M$ over $T$ (as $p_{1}^{*}(F)$ is a line bundle).  Furthermore, $\Jb_{g,n}^{d, \text{pre}}(A,M)  \to \Jb_{g,n}^{d, \text{pre}}(A,M) $ has the property that the image of $\mathbb{G}_{m}(T) \subset \Aut(C, p_1, \dots, p_n; F)$ is mapped to the identity in $\Aut(C, p_1, \dots, p_n; F \otimes \pi_{T}^{*} p_1^{*}(F)^{-1})$ (a scalar $g \in \mathbb{G}_{m}(T)$ acts by $g$ on $F$, by $g^{-1}$ on $\pi_{T}^{*} p_{1}^{*}(F)^{-1}$, so by $g g^{-1} = 1$ on the tensor product).  By the universal property of rigidification the morphism $\Jb_{g,n}^{d, \text{pre}}(A,M)  \to \Jb_{g,n}^{d, \text{pre}}(A,M) $ factors as $\Jb_{g,n}^{d, \text{pre}}(A,M)  \to \Jb_{g,n}^d(A,M)\to \Jb_{g,n}^{d, \text{pre}}(A,M)  $, and $\Jb_{g,n}^d(A,M)\to \Jb_{g,n}^{d, \text{pre}}(A,M) $ defines the desired section.
\end{proof}

\begin{remark} \label{Remark: TautologicalNotUnique}
	The tautological family $F_{\text{tau}}$ is not uniquely determined.   Given a tautological family $F_{\text{tau}}$ and a line bundle $L$ on $\Jb_{g,n}^d(A,M)$, the sheaf $F_{\text{tau}} \otimes \pi^{*}(L)$ is also a tautological family.  However, every tautological family is of the form  $F_{\text{tau}} \otimes \pi^{*}L$ for some line bundle $L$ on $\Jb_{g,n}^d(A,M)$ by the Seesaw theorem.
\end{remark}

We now construct the theta divisor in degree $d=g-1$ as  the determinant of the cohomology of $F_{\text{tau}}$.  Recall the more general  construction of the determinant of an element of the derived category.  Generalizing earlier work with Mumford, Knudsen proved that the rule that assigns to a  bounded complex $\mathcal{E}$ of vector bundles  on $\Jb_{g,n}^{g-1}(A,M)$ the line bundle $\operatorname{det}(\mathcal{E}) := \bigotimes (\bigwedge^{\text{max}} \mathcal{E}^{i})^{(-1)^{i}}$ extends to a rule that assigns an isomorphism of line bundles to a quasi-isomorphism of perfect complexes \cite[Theorem~2.3]{knudsen}, so the determinant of an object in the bounded derived category is defined. (See also \cite[Section~6.1]{esteves01} for a more explicit approach in the special case of a family of curves, the case of current interest.) The  derived pushforward $\mathbb{R} \pi_{*} F_{\text{tau}}$ of a tautological family is an element of the bounded derived category by the finiteness theorem \cite[Theorem~8.3.8]{illusie}, so in particular, its determinant $\operatorname{det}(\mathbb{R} \pi_{*} F_{\text{tau}})$ is defined.

The inverse line bundle $\operatorname{det}(\mathbb{R} \pi_{*} F_{\text{tau}})^{-1}$ admits a distinguished nonzero global section that is constructed as follows.  The morphism $\pi \colon \Cb_{g,n} \to \Mb_{g,n}$ has relative cohomological dimension $1$, so $\mathbb{R}\pi_{*} F_{\text{tau}}$ can be represented by a $2$-term complex of vector bundles $\mathcal{E}^{0} \stackrel{d}{\longrightarrow} \mathcal{E}^{1}$. The generic fiber of this complex computes the cohomology of a degree $g-1$ sheaf, so it has Euler characteristic zero (by the Riemann--Roch formula).  We deduce that $\operatorname{rank} \mathcal{E}^{0}=\operatorname{rank} \mathcal{E}^{1}$, and so the top exterior power $\operatorname{det}(d) := \bigwedge^{\text{max}}(d)$ is a global section of
\begin{align*}
	\ShHom\left( \operatorname{det} \mathcal{E}^{0},  \operatorname{det} \mathcal{E}^{1}\right) 	=& \left(\operatorname{det} \mathcal{E}^{0}\right)^{-1} \otimes   \operatorname{det} \mathcal{E}^{1} \\
																	=& \operatorname{det}\left(\mathbb{R}\pi_{*}F_{\text{tau}}\right)^{-1}.
\end{align*}
A direct computation shows that  $\operatorname{det}(d) \in H^{0}(\Mb_{g,n}, \operatorname{det}(\mathbb{R}\pi_{*}F_{\text{tau}})^{-1} )$ is independent of the choice of complex $\mathcal{E}$ (i.e.~that $\operatorname{det}(d)$ is preserved by isomorphisms induced by quasi-isomorphisms; see \cite[Observation~43]{esteves01}).

The line bundle $\operatorname{det}(\mathbb{R} \pi_{*} F_{\text{tau}})$ is uniquely determined even though $F_{\text{tau}}$ is not:
\begin{lemma} \label{Lemma: ThetaDivisorDefined}
	If $F_{\text{tau}}$ and $G_{\text{tau}}$ are two tautological families on $\Jb_{g,n}^{g-1}(A,M)$, then
	\[
		\operatorname{det}(\mathbb{R} \pi_{*} F_{\text{tau}}) = \operatorname{det}(\mathbb{R} \pi_{*} G_{\text{tau}}),
	\]
	and this  identifies
	\[
		\operatorname{det}(d) \in H^{0}(\Mb_{g,n},  \operatorname{det}(\mathbb{R} \pi_{*} F_{\text{tau}})^{-1})
	\]
	 with
	 \[
	 	\operatorname{det}(e) \in H^{0}(\Mb_{g,n},  \operatorname{det}(\mathbb{R} \pi_{*} G_{\text{tau}})^{-1}).
	\]
\end{lemma}
\begin{proof}
	By Remark~\ref{Remark: TautologicalNotUnique}, $G_{\text{tau}}=F_{\text{tau}} \otimes \pi^{*}(N)$ for some line bundle $M$ on $\Jb_{g,n}^{d}(A,M)$. The result follows from the projection property of the determinant \cite[Proposition~44(3)]{esteves01}.
\end{proof}
In the lemma, it is important that $d=g-1$ for otherwise the conclusion would fail to hold.  Indeed,  loc.~cit.~shows $\operatorname{det}(\mathbb{R} \pi_{*} G_{\text{tau}})^{-1}$ equals $N^{\otimes d+1-g} \otimes \operatorname{det}(\mathbb{R} \pi_{*} F_{\text{tau}})^{-1}$.

\begin{definition} \label{Def: ThetaAM}
	For $n \geq 1$, the \textbf{theta divisor} $\ThDiv(A,M) \subset \Jb_{g,n}^{g-1}(A,M)$ is the effective Cartier divisor defined by $(\operatorname{det}(\mathbb{R}\pi_{*}F_{\text{tau}})^{-1}, \operatorname{det}(d))$ for $\mathcal{E}^{0} \stackrel{d}{\longrightarrow} \mathcal{E}^{1}$ a $2$-term complex of vector bundles that represents $\mathbb{R}\pi_{*} F_{\text{tau}}$.  The \textbf{theta divisor Chow class} $\ThDivCl(A,M) \in \Chow^{1}(\Jb_{g,n}^{g-1}(A,M))$ is  the fundamental class $[\ThDiv(A,M)]$.
\end{definition}

We conclude this discussion by describing the property that defines the theta divisor as a set.
\begin{lemma} \label{Lemma: ThetaSubschemeExists}
	For $n \geq 1$, the theta divisor $\ThDiv(A,M)$ is supported on the locus of points $(C, p_1, \dots, p_n; F) \in \Jb^{g-1}_{g,n}(A,M)$ with $H^{0}(C,F) \ne 0$.
\end{lemma}
\begin{proof}
	Fix a $2$-term complex of vector bundles $\mathcal{E}^{0} \stackrel{d}{\longrightarrow} \mathcal{E}^{1}$ that represents $\mathbb{R}\pi_{*} F_{\text{tau}}$, so that $\ThDiv(A,M) = \{ \operatorname{det}(d)=0 \}$. Given a point  $(C, p_1, \dots, p_n; F)$, write
	\begin{displaymath}
			\mathcal{E} \otimes k(\text{point}) := \mathcal{E}^{0} \otimes k(\text{point}) \stackrel{d \otimes 1}{\longrightarrow} \mathcal{E}^{1} \otimes k(\text{point})
	\end{displaymath}
	for the fiber of  $\mathcal{E}^{0} \to \mathcal{E}^{1}$ at $(C, p_1, \dots, p_n; F)$.  The point $(C, p_1, \dots, p_n; F)$ lies in $\ThDiv(A,M)$ if and only if the complex $\mathcal{E} \otimes k(\text{point})$ has nonzero cohomology, and because the formation of $\mathbb{R} \pi_{*} F_{\text{tau}}$ commutes with base change \cite[Theorem~8.3.2]{illusie}, the cohomology groups of $\mathcal{E} \otimes k(\text{point})$ are $H^{0}(C, F)$ and $H^{1}(C, F)$.
\end{proof}

We now restrict our constructions of $\Jb_{g,n}^d(A,M)$  and of the theta divisor $\ThDiv(A,M)$ to the treelike locus $\mathcal{M}_{g,n}^{\text{TL}} \subseteq \Mb_{g,n}$. 

\begin{definition} \label{Def: Jbar}
Given $\phi \in V_{g,n}^{\text{TL}}$,  we say that a family  of rank~$1$, torsion-free sheaves  of degree $g-1$ on a family of treelike curves is \textbf{$\phi$-(semi)stable} if the fibers are $\phi$-(semi)stable.  We define  $\Jb_{g,n}^{\text{pre}}(\phi)$ and  the \textbf{$\phi$-compactified universal Jacobian} $\Jb_{g,n}(\phi)$ in analogy with Definition~\ref{Def: moduliStacks}.
\end{definition}

From the results we have proved for $\Jb^{d}_{g, n}(A, M)$, we  deduce the following result about $\Jb_{g,n}(\phi)$.

\begin{corollary} \label{Cor: JbExists}
	Assume $n \geq 1$ and $\phi \in V_{g,n}^{\text{TL}}$ is nondegenerate.  Then the forgetful morphism $\Jb_{g,n}(\phi) \to \Mm{g}{n}^{\text{TL}}$ is representable, proper, and flat. In particular, $\Jb_{g,n}(\phi)$ is a separated Deligne-Mumford stack. Furthermore, $\Jb_{g,n}(\phi)$ is $k$-smooth.
\end{corollary}
\begin{proof}
	By combining Lemma~\ref{Lemma: Transitivity} and Lemma~\ref{existsAM} we deduce that there exists $A$ and $M$ such that $\phi(A, M)$ belongs to $\mathcal{P}(\phi)$, so the restriction to the treelike locus $\Jb^{g-1}_{g,n}(A, M)| \Mm{g}{n}^{\text{TL}}$ equals $\Jb_{g,n}(\phi)$.  The result thus follows from Proposition~\ref{existence1} and Lemma~\ref{flatandsmooth}.
\end{proof}

We proved Corollary~\ref{Cor: JbExists} using Lemma~\ref{Lemma: Transitivity}, and from the same lemma, we also deduce the following corollary.

\begin{corollary} \label{uniqueiso}
	Given nondegenerate stability parameters $\phi_1, \phi_2 \in V_{g, n}^{\text{TL}}$, there exists a unique line bundle $M \in W_{g,n}$ such that $F \mapsto F \otimes M$  defines an isomorphism
	\[
		\Jb_{g,n}(\phi_1) \to \Jb_{g,n}(\phi_2).
	\]
\end{corollary}
\begin{proof}
	This is a restatement of  Lemma~\ref{Lemma: Transitivity}.
\end{proof}

We will need the following property of the fibers of the forgetful morphism.

\begin{lemma} \label{fiberirreducible}
	For $n \geq 1$ and $\phi \in V_{g,n}^{\text{TL}}$ nondegenerate, the fibers of the forgetful morphism $\Jb_{g,n}(\phi) \to \Mm{g}{n}^{\text{TL}}$ are irreducible.
\end{lemma}
\begin{proof}
	This follows from Lemma~\ref{Lemma: ComponentGroup}.  In a fiber of $\Jb_{g,n}(\phi) \to \Mm{g}{n}^{\text{TL}}$, the locus of line bundles of fixed multidegree is $k$-smooth and connected, hence irreducible.  (The locus is a torsor for the generalized Jacobian, a semiabelian variety).  There is a unique $\phi$-stable multidegree of a line bundle by Lemma~\ref{Lemma: ComponentGroup}, so we conclude that the line bundle locus in a fiber is irreducible.  But the line bundle locus is also dense in its fiber because it is the $k$-smooth locus by  e.g.~the description of the completed local ring in \cite[Corollary~3.17]{kass12b} (where the ring appears as $R_1$). We conclude that the fiber is irreducible.
\end{proof}

Here we define the theta divisor in $\Jb_{g,n}(\phi)$.

\begin{definition} \label{Def: Theta} For $\phi \in V_{g,n}^{\text{TL}}$, we define $\ThDiv(\phi) \subset \Jb_{g,n}(\phi)$ to be the restriction of $\ThDiv(A,M)$ (from Definition~\ref{Def: ThetaAM}) to the treelike locus, for any $(A,M)$ such that $\phi(A,M) \in \mathcal{P}(\phi)$. We define $\ThDivCl(\phi) \in \Chow^{1}(\Jb_{g,n}(\phi))$ as  the fundamental class $[\ThDiv(\phi)]$.
\end{definition}

We conclude this section by characterizing when $\ThDiv(\phi) \to \Mm{g}{n}^{\text{TL}}$ is flat.

\begin{lemma} \label{Lemma: WhenIsPhiFlat?}
	Let $n \geq 1$. Given a nondegenerate $\phi_0 \in  V_{g,n}^{\text{TL}}$,  the natural projection   $\ThDiv(\phi) \to \Mm{g}{n}^{\text{TL}}$ is flat  if and only if $\phi_{\text{can}} \in \overline{\mathcal{P}}(\phi_0)$ (the closure of the unique stability polytope that contains $\phi_0$).
\end{lemma}
\begin{proof}
	Since $\ThDiv(\phi)$ is an effective divisor on $\Jb_{g,n}(\phi)$ and $\Jb_{g,n}(\phi) \to \mathcal{M}^{\text{TL}}_{g,n}$ is flat, the morphism $\ThDiv(\phi) \to \mathcal{M}^{\text{TL}}_{g,n}$ is flat if and only if $\ThDiv(\phi)$ does not contain a fiber of $\Jb_{g,n}(\phi) \to \Mm{g}{n}^{\text{TL}}$.  The theta divisor  $\ThDiv(\phi)$ is supported on the locus of sheaves that admit a nonzero section (Lemma~\ref{Lemma: ThetaSubschemeExists}), so to complete the proof, we need to show that, on every stable treelike curve $C$, there is a $\phi$-stable sheaf $F$ with $H^{0}(C, F)=0$ if and only if $\phi_{\text{can}} \in \overline{\mathcal{P}}(\phi)$.  From \cite[Proposition~3.6]{alexeev04} combined with \cite[Lemma 2.1]{beauville}, we deduce that, on a given $C$, there exists a $\phi$-stable line bundle $L$ with $H^{0}(C, L)=0$ if and only if every $\phi$-stable line bundle is $\phi_{\text{can}}$-semistable.  Thus the result follows from Lemma~\ref{Lemma: StabilityPolytope}.
\end{proof}

\subsection{Concluding remarks} \label{Subsection: Remarks}
    We conclude with some remarks, beginning with remarks about the families $\Jb_{g,n}(\phi)$ and their relation to families already existing in the  literature.  By definition $\Jb_{g,n}(\phi)$ is the moduli space of $\phi$-semistable rank~$1$, torsion-free sheaves.  Our definition of $\phi$-semistability (Definition~\ref{Def: StabilityInFamilies}) is a modification of the definition given by Oda--Seshadri \cite{oda79} (for degree $0$ rank~$1$, torsion-free sheaves on a single nodal curve), and our proof of Corollary~\ref{Cor: JbExists} shows that $\phi$-semistability can be (non-canonically) identified with slope semistablity in the sense of \cite{simpson94}.  In \cite{esteves01}, Esteves defined a stability condition, called quasi-stability, for sheaves on a family of curves.  Melo recently used Esteves' work to construct compactifications of $\J_{g,n}$ over $\Mb_{g,n}$, and showed that every $\Jb_{g, n}(\phi)$ can be realized as the restriction of one of her compactifications to $\Mm{g}{n}^{\text{TL}}$ \cite[Proposition~4.17]{melo16}.

    Earlier in \cite{melo11} Melo constructed a compactified universal Jacobian over $\Mb_{g,n}$.  Her compactification is different from the ones studied in this paper as e.g.~it is not  a Deligne--Mumford stack  (as the hypothesis to \cite[Proposition~8.3]{melo11} fails; her compactification is also not a moduli stack of torsion-free sheaves on stable curves, but the authors expect one can identify it with such a stack by an argument similar to \cite[Theorem~10.3.1]{pand96}).  Melo's paper builds upon a large body of work on constructing compactifications over $\Mb_{g,0}$ \cite{caporaso94, pand96, jarvis00, caporaso08, melo09}.  The compactifications constructed by Caporaso and Pandharipande have special significance because of a relation to moduli of stable pairs (in the sense of  Alexeev) that we explain in Section~\ref{Subsect: StablePair}.

	To describe the compactifications constructed by Caporaso and Pandharipande and their relations to the compactifications studied in this paper, consider the fiber $\overline{J}_{C}$ of their family over a curve $C$ that has two smooth irreducible components $C^{+}$ and $C^{-}$, each with positive genus.  This fiber is described in terms of  $\phi_{\text{can}}$-semistable sheaves in \cite[Theorem~10.3.1]{pand96}. There are no $\phi_{\text{can}}$-stable line bundles, and the $\phi_{\text{can}}$-semistable line bundles are the line bundles of bidegree $(i-1, g-i)$ and $(i, g-1-i)$.  The coarse moduli space of $\phi_{\text{can}}$-semistable sheaves is not  $\overline{J}_{C}$.  Rather $\overline{J}_{C}$ is the quotient of the coarse moduli space by the natural action of $\Aut(C)$.  (The appearance of the quotient by $\Aut(C)$ is related to the fact that in \cite{pand96, caporaso94} the authors construct a scheme with a morphism to the coarse scheme of $\Mb_{g, 0}$ rather than an algebraic stack with a morphism to $\Mb_{g, 0}$.)

For $\phi$ nondegenerate, the fiber $\overline{J}_{C}(\phi)$ of $\Jb_{g, n}(\phi) \to \Mb_{g, n}$ over $C$ is the moduli space of bidegree $(d_-, d_+)$ line bundles for $d_{-}, d_{+}$ depending on $\phi$.  For any two nondegenerate stability parameters $\phi_1$ and $\phi_2$, Corollary~\ref{uniqueiso} describes an isomorphism between $\overline{J}_{C}(\phi_1)$ and $\overline{J}_{C}(\phi_2)$, and \cite[7.2]{caporaso94} shows that the Caporaso--Pandharipande fiber $\overline{J}_{C}$ is isomorphic to the quotient of $\overline{J}_{C}(\phi_1)$ by $\Aut(C)$.

Corollary~\ref{uniqueiso} shows more generally that, as $\phi$ varies, the isomorphism class of $\Jb_{g,n}(\phi)$ does not vary.  The authors expect this result is an artifact of the fact that we study extensions of $\J_{g,n} \to \Mm{g}{n}$ to a family over $\Mm{g}{n}^{\text{TL}}$ rather than over all of $\Mb_{g,n}$, Indeed, the examples in \cite{melo14} suggest that there are many different schemes that extend the universal Jacobian to a family of moduli spaces over all of $\Mb_{g,n}$.

This brings us to the second topic of discussion: the stability space $V_{g,n}^{\text{TL}}$.  We have defined $V_{g,n}^{\text{TL}}$ so that it controls families over $\Mm{g}{n}^{\text{TL}}$. A consequence of  Theorem~\ref{Thm: wc} in Section~\ref{Section: Wallcrossing} is that the decomposition of $V_{g,n}^{\text{TL}}$ defined by the variation of the theta divisor \emph{essentially} coincides with the stability polytope decomposition, the only difference being that the theta divisor is constant on all the (finitely many) polytopes that contain $\phi_{\text{can}}$ in their closures (a consequence of Lemma~\ref{Lemma: WhenIsPhiFlat?}).

The authors believe that $\Mm{g}{n}^{\text{TL}}$ is the largest open substack $\mathcal{W} \subset \Mb_{g,n}$ that is a union of topological strata, with the property that the different theta divisors are essentially in bijection with the different extensions of $\J_{g,n}$ to a family over $\mathcal{W}$.

\section {Wall-crossing formula for the theta divisor} \label{Section: Wallcrossing}

In this section we study how the theta divisor class $\ThDivCl(\phi)$ varies with $\phi$ by proving a wall-crossing formula. The main result is Theorem~\ref{Thm: wc}, which we prove by applying the Grothendieck--Riemann--Roch theorem to a test curve. We then show in Corollary~\ref{Cor: wctranslated} how to deduce a wall-crossing formula for the translated of the theta divisor. We prove in Corollary~\ref{Cor: translatedtheta} that the translates of the theta divisor, together with line bundles pulled back from $\Mb_{g,n}$, form a set of generators for the Picard group of $\Jb_{g,n}$. Therefore our results describe how all codimension $1$ classes of $\Jb_{g,n}(\phi)$ vary in $\phi$. From now on we  assume $n \geq 1$. 

In the previous section we defined a stability space $V_{g,n}^{\text{TL}}$ (Definition~\ref{Def: StabilityInFamilies}) endowed with a stability polytope decomposition (Definition  \ref{Definition: PolytopesFamilyCurves}); for any nondegenerate $\phi \in V_{g,n}^{\text{TL}}$ we then defined a universal $\phi$-compactified Jacobian $\Jb_{g,n}(\phi)$ over $\mathcal{M}_{g,n}^{\text{TL}}$ (Definition~\ref{Def: Jbar}), we proved it is a $k$-smooth Deligne--Mumford stack (Corollary~\ref{Cor: JbExists}) and then defined a theta divisor $\ThDiv(\phi) \subset \Jb_{g,n}(\phi)$ as an effective Cartier divisor (Definition~\ref{Def: Theta}). We proved in Corollary~\ref{uniqueiso} that for any two nondegenerate $\phi_1, \phi_2 \in V_{g,n}^{\text{TL}}$, the compactified universal Jacobians  $\Jb_{g,n}(\phi_1)$ and $\Jb_{g,n}(\phi_2)$ are isomorphic, and the isomorphism is unique if it is chosen among those of the kind $F \mapsto F \otimes L$ for $L \in W_{g,n}$, (the subgroup of $\Pic^{0}(\Cb_{g,n}/\Mm{g}{n}^{\text{TL}})$ generated by the components $\mathcal{O}(\mathcal{C}^{\pm}_{i, S})$). In this section we will always identify $\Jb_{g,n}(\phi_1)$ and $\Jb_{g,n}(\phi_2)$ by means of this unique isomorphism.

To describe the difference $\ThDivCl(\phi_2) - \ThDivCl(\phi_1)$ we can assume that $\overline{\calP}(\phi_1)$ and  $\overline{\calP}(\phi_2)$ share a common facet, namely that $\overline{\calP}(\phi_1) \cap \overline{\calP}(\phi_2)$ generates a stability hyperplane in  $V_{g,n}^{\text{TL}}$ that we call $H=H(i,S,d)$  (see Lemma~\ref{Lemma: HyperplanesByTwoComponents}).  Therefore $\calP(\phi_1)$ and $\calP(\phi_2)$ have the same $\Gamma(i',S')$ coordinates except when $(i',S') = (i,S)$, and we fix the sign convention so that
\begin{equation} \label{signconvention}
\calP(\phi_2)(i,S) = \calP(\phi_1)(i,S) + (1, -1).
\end{equation}
 In more concrete terms, on any treelike curve not in $\Delta_{i,S}$, the $\phi_1$-stable sheaves coincide with the $\phi_2$-stable sheaves, and for a general fiber over $\Delta_{i,S}$ the $\phi_1$-stable sheaves are the line bundles of bidegree $(d, g-1-d)$, and the $\phi_2$-stable sheaves are the line bundles of bidegree  $(d+1,   g-2-d)$.

We can now formulate the wall-crossing formula for the theta divisor class.
\begin{theorem} \label{Thm: wc} Let $\phi_1$ and $\phi_2$ be nondegenerate stability parameters that belong to two stability polytopes of $V_{g,n}^{\text{TL}}$ that share a common facet, let $H(\Gamma(i,S), d) =: H(i,S,d)$ be the stability hyperplane containing that facet, and assume that the two stability polytopes are related by the sign convention \eqref{signconvention}. Then
\begin{align} \label{wc}
\ThDivCl(\phi_2) - \ThDivCl(\phi_1)  = &\left(\left\lfloor \phi_2^+(i,S)+ \frac{1}{2}\right\rfloor-i \right) \cdot \delta_{i,S} \notag \\   = & \left(d+1-i\right) \cdot \delta_{i,S} .
\end{align}
\end{theorem}
\noindent (As is customary, we have written $\delta_{i,S}$ for the pullback of the boundary divisor class along $\Jb_{g,n} \to \mathcal{M}_{g,n}^{\text{TL}}$).
\begin{proof}
 Choosing tautological bundles $\Tau(\phi_1)$ and $\Tau(\phi_2)$ as in Lemma \ref{Lemma: SectionExists}, we have \[F_{\text{tau}}(\phi_2) \cong F_{\text{tau}}(\phi_1) \otimes \mathcal{O}(\mathcal{C}_{i,S}^-).\]
 By Definition \ref{Def: Theta}, the left-hand side of \eqref{wc} is the first Chern class of the line bundle
 \begin{equation} \label{elldefinition}
 L:= \left(\det( \mathbb{R} \pi_* \Tau(\phi_1) \otimes \mathcal{O}(\mathcal{C}_{i,S}^-)) \right) ^{-1} \otimes \det( \mathbb{R} \pi_* \Tau(\phi_1)).
 \end{equation}
We claim that the line bundle $L$ is the pullback of $\mathcal{O}(\Delta_{i,S})^{\otimes c}$ for some $c \in \mathbb{Z}$. Indeed, over the complement of $\Delta_{i,S}$ the restriction of $\mathcal{O}( \mathcal{C}^-_{i,S})$ is trivial.  Since the formation of the determinant of cohomology commutes with base change, the restriction of $L$ to $\Jb_{g,n} - \Delta_{i,S}$ is also trivial,    which combined with the fact that the fibers of $\Jb_{g,n} \to \mathcal{M}_{g,n}^{\text{TL}}$ are irreducible (Lemma~\ref{fiberirreducible}), implies our claim.

 The integer $c$ is determined by computing the other two integers in the equality
 \begin{equation} \label{testexplanation}
 c \cdot \deg \left( \mathcal{O}(\Delta_{i,S})  _{| T} \right)  = \deg L_{| T}.
 \end{equation}

 Let $(\pi_T \colon C \to T, p_1, \ldots, p_n)$ be a {\bf test curve} for $\mathcal{M}_{g,n}^{\text{TL}}$: the pullback to a $k$-smooth curve $T \to \mathcal{M}_{g,n}^{\text{TL}}$ of the universal curve $\pi \colon \Cb_{g,n} \to \mathcal{M}_{g,n}^{\text{TL}}$ and of the universal sections.  Every such $T \to \mathcal{M}_{g,n}^{\text{TL}}$ lifts to a morphism  $T \to \Jb_{g,n}(\phi_1)$; equivalently there exists a family of $\pi_T$-fiberwise $\phi_1$-stable sheaves $F$ on $C$. This is a consequence of the existence of a section of the forgetful morphism $\Jb_{g,n}(\phi_1) \to \mathcal{M}_{g,n}^{\text{TL}}$, and we will construct several such sections in the beginning of Section~\ref{PullbackTheta},  see \eqref{muller} and \eqref{muller2}.

 From Proposition~\ref{Prop: TestcurveforJbar} below (a Grothendieck--Riemann--Roch calculation) we deduce that the right-hand side of \eqref{testexplanation} equals
\begin{equation} \label{prop1explanation}
\deg L_{| T}= - \deg\left( \pi_{T *} \left( \left( \ch F (\mathcal{C}^-_{i,S})- \ch(F) \right) \cap \ \td C \right) \right).
\end{equation}

In Construction~\ref{construction} we produce an explicit test curve for $\mathcal{M}_{g,n}^{\text{TL}}$ whose intersection with the divisor $\Delta_{i,S}$ is the class of one point:
\begin{equation}\label{transversalint} \deg \left( \mathcal{O}(\Delta_{i,S})_{|T} \right)= 1.\end{equation}  We then prove in Lemma~\ref{Lemma: IntersectionCalculation} that, on this test curve, Equation~\eqref{prop1explanation} becomes
\begin{equation} \label{final}
 \deg L_{| T} = -\deg\left( F_{| \mathcal{C}^-_{i,S}}\right) +(g-i) = -\left(g-1-d \right) + (g-i) = \left(d+1-i \right).
\end{equation}
Combining Equation~\eqref{testexplanation} with  \eqref{transversalint} and \eqref{final} gives $c=d+1-i$, which concludes the proof of Theorem~\ref{Thm: wc}.
 \end{proof}

\begin{remark} Using the classical Riemann--Roch formula, we can express the coefficient of $\delta_{i,S}$ in Formula~\eqref{wc} as the Euler characteristic of a line bundle.

Writing $C = C^{+} \cup C^{-}$ for a general fiber of $\pi^{-1}(\Delta_{i,S}) \to \Delta_{i,S}$, we have the equalities
\begin{displaymath}
d+1-i =  \deg (\Tau(\phi_1)_{| {C}^+}) + 1-i=   \chi({C}^+, \Tau(\phi_1)) .
\end{displaymath}
Formula~\eqref{wc} can therefore be written in the suggestive form
\begin{equation}
\ThDivCl(\phi_2) - \ThDivCl(\phi_1) =    \chi({C}^+, \Tau(\phi_1)) \cdot \delta_{i,S}=  - \chi(C^-, \Tau(\phi_2))\cdot \delta_{i,S}.
\end{equation}
\end{remark}

We now present some easy corollaries of Formula~\eqref{wc}.

As a consequence of Lemma~\ref{Lemma: WhenIsPhiFlat?}, when $\phi_{\text{can}} \in \overline{\calP}(\phi_1), \overline{\calP}(\phi_2)$, the theta divisors $\ThDiv(\phi_1)$ and $\ThDiv(\phi_2)$ are flat on $\mathcal{M}_{g,n}^{\text{TL}}$, therefore both are the closure of the theta divisor on the universal Jacobian over smooth curves, hence they coincide. Theorem~\ref{Thm: wc} provides the converse implication (assuming as usual that $\overline{\calP}(\phi_1)$ and $\overline{\calP}(\phi_2)$ share a common facet).
\begin{corollary} \label{cor: closurecan} If  $\ThDiv(\phi_1)$ equals  $\ThDiv( \phi_2)$, then $\phi_{\text{can}}$ belongs to $\overline{\calP}(\phi_1)$ and $\overline{\calP}(\phi_2)$.
\end{corollary}

If $\ThDiv(\phi)$ is flat on $\mathcal{M}_{g,n}^{\text{TL}}$, in particular it is reduced. In the following corollary we determine all stability parameters whose associated theta divisor is reduced.
\begin{corollary} \label{Corollary: Nonreduced}
	If $\phi \in V_{g,n}^{\text{TL}}$ is nondegenerate, then $\ThDiv(\phi)$ is reduced if and only if there exists a stability polytope $\mathcal{Q}$ such that $\overline{\mathcal{P}}(\phi) \cap \overline{\mathcal{Q}} \ne \emptyset$ and $\phi_{\text{can}} \in \overline{\mathcal{Q}}$.
\end{corollary}

\begin{proof} If we let $\ThDiv(\phi)_{\text{red}}$ be the reduced structure on the theta divisor, we have the inclusion of subschemes $\ThDiv(\phi)_{\text{red}} \subseteq \ThDiv(\phi)$. We claim that $\ThDiv(\phi)$ consists of an irreducible component dominant over $\mathcal{M}_{g,n}^{\text{TL}}$, and of other components supported on the inverse image of the boundary divisors $\Delta_{i,S}$. Indeed, the restriction of $\ThDiv(\phi)$ to $\mathcal{M}_{g,n}$ is irreducible (because the theta divisor of a smooth curve is irreducible), so its closure in $\Jb_{g,n}(\phi)$ is the unique irreducible component that dominates $\mathcal{M}_{g,n}^{\text{TL}}$. If $Z$ is any other component of $\ThDiv(\phi)$, then $Z$ has codimension $1$ in $\Jb_{g,n}(\phi)$ because $\ThDiv(\phi)$ is an effective Cartier divisor, and because the image of $Z$ in $\mathcal{M}_{g,n}^{\text{TL}}$ does not intersect $\mathcal{M}_{g,n}$, the dimension of a general fiber of $Z$ equals $g$. We deduce that every fiber of $Z$ has dimension $g$; because every fiber of $\Jb_{g,n}(\phi) \to \mathcal{M}_{g,n}^{\text{TL}}$ has dimension $g$ and  by Lemma~\ref{fiberirreducible} it is irreducible, we conclude that in fact $Z$ is supported over the inverse image of some boundary divisor $\Delta_{i,S}$.

The component dominant over $\mathcal{M}_{g,n}^{\text{TL}}$ is isomorphic to $\ThDiv(\phi_0)$ for any $\phi_0 \in \mathcal{Q}$ (for $\mathcal{Q}$ any stability polytope such that $\phi_{\text{can}} \in \overline{\mathcal{Q}}$), and the latter is reduced by Lemma~\ref{Lemma: WhenIsPhiFlat?}. We can then write an equality of divisor classes
\[
\ThDivCl(\phi) = \ThDivCl(\phi)_{\text{red}} + \sum_{(i,S)} a_{i,S} \cdot \delta_{i,S}
\]
(where once again we have written $\delta_{i,S}$ for the pullback of the boundary divisor class by $\Jb_{g,n}(\phi_i) \to \mathcal{M}_{g,n}^{\text{TL}}$).
If $\ThDiv(\phi)$ is reduced, then so are all its irreducible components, which implies that all coefficients $a(i,S)$ are either $0$ or $1$. The converse implication is provided in Lemma~\ref{Lemma: inclusiondivisors}.

It follows from our main Theorem~\ref{Thm: wc} that the coefficients $a_{i,S}$ are either $0$ or $1$ precisely when $\overline{\mathcal{P}}(\phi)$ and $\overline{\mathcal{Q}}$ have a common facet, where $\overline{\mathcal{Q}}$ is any of the polytopes containing $\phi_{\text{can}}$ in its closure.
\end{proof}

In the proof of Corollary~\ref{Corollary: Nonreduced} we used the following lemma, for which we provide a proof as we could not find a suitable reference.
\begin{lemma} \label{Lemma: inclusiondivisors} Let $\overline{\mathcal{X}}$ be a $k$-smooth proper Deligne--Mumford stack, and $D_1$ and $D_2$ be two effective divisors with $D_1$ a closed substack of $D_2$. Then the inclusion  induces an isomorphism of $D_1$ and $D_2$ if and only if the divisor classes $[D_1]$ and $[D_2]$ are equal.
\end{lemma}
\begin{proof}
 We first reduce to the setting of divisors on a projective scheme by picking an \'{e}tale cover $\Mb \to \overline{\mathcal{X}}$ with $\Mb$ smooth and projective.   We have an exact sequence of coherent sheaves
	\[
		0 \to \mathcal{O}_{D_2}(-D_1) \to \mathcal{O}_{D_2} \to \mathcal{O}_{D_1} \to 0.
	\]
	Now pick an ample line bundle $A$ on $\Mb$ and consider the Hilbert polynomials $p_1(t)$ and $p_{2}(t)$ respectively associated to $D_1$ and $D_2$.  The polynomial $p_i$ has degree equal to $\dim ({\Mb})-1$ and leading term equal to the degree of $D_i$ (computed with respect to $A$).  Since $D_1$ is linearly equivalent to $D_2$ by hypothesis, $\mathcal{O}_{D_1}$ and $\mathcal{O}_{D_2}$ have the same degree, and so $p_2-p_1$ has degree strictly less than $\dim ({\Mb})-1$.  By additivity $p_2-p_1$ equals the Hilbert polynomial of $\mathcal{O}_{D_{2}}(-D_1)$, and we conclude that this last Hilbert polynomial has degree strictly less than $\dim ({\Mb})-1$.

	This is only possible if $\mathcal{O}_{D_2}(-D_1)$ equals zero.  Indeed, $\mathcal{O}_{D_2}(-D_1)$ is locally principal, so if $\mathcal{O}_{D_2}(-D_1)$ was nonzero, then its support would have dimension  $\dim (\Mb)-1$.  Since $\mathcal{O}_{D_2}(-D_1)=0$, the inclusion of $D_1$ in $D_2$ is an isomorphism.
\end{proof}

We conclude this section by proving the auxiliary results that we used to prove Theorem~\ref{Thm: wc}.
\begin{proposition} \label{Prop: TestcurveforJbar} Let $(\pi_T \colon C \to T, p_1, \ldots, p_n, F)$ be a test curve for $\Jb_{g,n}(\phi)$. Then the following equality
\[
(c_1 \left(\mathbb{R} \pi_* F (\mathcal{C}_{i,S}^-) \right)- c_1 ( \mathbb{R} \pi_* F) ) \cap [T] =  \pi_{T *} \left( \left( \ch F (\mathcal{C}^-_{i,S})- \ch(F) \right) \cap \ \td C \right)
\]
holds in the Chow group of $0$-cycles on $T$.
\end{proposition}

 \begin{proof}

The $0$-th and $1$-st higher pushforwards of $F$ and of $F (\mathcal{C}^-_{i,S})$ under $\pi_T$ are sheaves of the same rank. Indeed, taking derived pushforwards commutes with base change, and the $0$-th and $1$-st cohomology of $F$ and $F (\mathcal{C}^-_{i,S})$ on the fiber of a geometric point in $T$ have the same dimension by the Riemann--Roch formula for curves, since the fiberwise degree is $g-1$. Therefore we have that both the degree-$0$ Chern characters
\begin{displaymath}
\ch_0 (\mathbb{R} \pi_* F), \quad \ch_0 (\mathbb{R} \pi_* F(\mathcal{C}^-_{i,S}))
\end{displaymath}
vanish, so we deduce the following equality in the Chow group of $0$-cycles on $T$:
\begin{equation}\label{pregrr}
c_1 \left(\mathbb{R} \pi_* F (\mathcal{C}_{i,S}^-) \right)- c_1 ( \mathbb{R} \pi_* F) ) \cap [T]=  \left( \ch  \mathbb{R} \pi_* F (\mathcal{C}^-_{i,S}) - \ch \mathbb{R} \pi_* F  \right) \cap \ \td T.
\end{equation}

Applying the Grothendieck-Riemann-Roch formula to $\pi_T$, we have
\begin{equation} \label{grr}
  \left( \ch  \mathbb{R} \pi_* F (\mathcal{C}^-_{i,S}) - \ch \mathbb{R} \pi_* F  \right) \cap \ \td T = \pi_{T *}  \left( \left( \ch F (\mathcal{C}^-_{i,S})- \ch(F) \right) \cap \ \td C \right),
\end{equation}
and the statement follows by combining Equations~\eqref{pregrr} and \eqref{grr}.
\end{proof}

\begin{construction} \label{construction} For each pair $(i,S)$ as in Notation~\ref{Definition: AdmissiblePair}, we construct a test curve $(\pi_T \colon C \to T, p_1, \ldots, p_n)$ whose intersection with $\Delta_{i,S}$ is the class of one point,
and whose general fiber is in $\Delta_{i, S \setminus \{1\}}$, as shown in Figure~\ref{figure: testcurve}. (The special cases $i=0$ and $|S|=2$ are left to the reader).

\begin{figure}[ht]
\centering

\begin{tabular}{cc}
\includegraphics[scale=0.4]{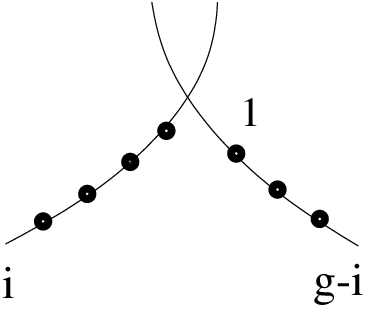} \ \ \ \ & \includegraphics[scale=0.4]{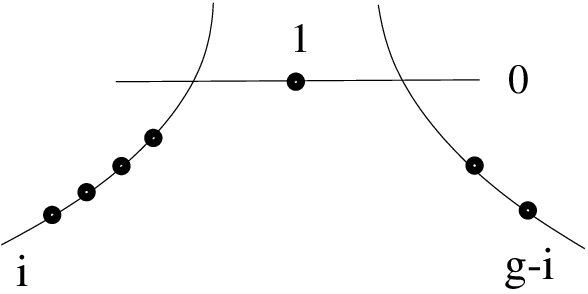}
\end{tabular}
\caption{The general fiber and the special fiber of the test curve}
\label{figure: testcurve}
\end{figure}

Fix a general genus $g-i$ pointed curve $(T, S^c \cup \{1 \})$ and a general genus $i$ pointed curve $(T', S \cup \{ \bullet\} \setminus \{ 1\} )$.   In $T \times T$ the diagonal intersects the locus \[T \times \{ p_k\colon  k \in S^{c} \cup \{1\}\}\] at the points $\{ ( p_k, p_k) \colon k \in S^{c} \cup \{ 1\} \}$, and we define the blow-up of these points to be  $\widetilde{C}_{1} \to T \times T$.  We then define $\widetilde{C}_{2}$ to be $T \times T'$.

The diagonal map $\Delta \colon T \to T \times T$ induces a morphism $s_1 \colon T \to \widetilde{C}_{1}$, and we define the morphism $s_2 \colon T \to \widetilde{C}_{2}$ as the constant $\bullet$ section of the first projection map.  We then define $C$ to be the following pushout (which exists by \cite[Theorem~5.4]{ferrand03})
\begin{displaymath}
	\begin{CD}
		T \cup T		 @>s_1 \cup s_2>>  	\widetilde{C}_1 \cup \widetilde{C}_{2} 	\\
		@VVV				@V\nu VV 							\\
		T			@>j >>	C.
	\end{CD}
\end{displaymath}
The projection onto the first component defines a morphism $\widetilde{C}_{1} \cup \widetilde{C}_{2} \to T$ that induces a morphism $\pi_T \colon C \to T$ by the universal property of pushouts.  The morphism $\pi_T$ inherits pairwise disjoint sections $p_1, \ldots, p_n$,
whose images lie in smooth locus of $\pi_{T}$.

 The family $\pi_T$ then defines a morphism $T \to \mathcal{M}_{g,n}^{\text{TL}}$ whose intersection with $\Delta_{i,S}$ is the class of one point (by construction the total space $C$ of the family is $k$-smooth at the unique node of type $\Delta_{i,S}$).
\end{construction}

\begin{lemma} On the test curve $(\pi_T \colon C \to T, p_1, \ldots, p_n)$ defined in Construction~\ref{construction}, we have\label{Lemma: IntersectionCalculation}
\begin{equation} \label{degreecomputation}
\deg\left( \pi_{T *} \left( \left( \ch F (\mathcal{C}^-_{i,S})- \ch(F) \right) \cap \ \td C \right) \right) = \deg\left( F_{| \mathcal{C}^-_{i,S}}\right) -(g-i)
\end{equation}
\end{lemma}

\begin{proof}

Since we are after the calculation of the degree of a $0$-cycle, for this proof we will work \emph{modulo numerical equivalence}. In the calculations we will follow standard notation: we omit writing fundamental classes, we write $(a,b)$ for the ruling on a product, and we write $[\pt]$ for the class of a point.

We claim that the Todd class of $C$ equals
\begin{equation} \label{two}
\td(C)= \nu_*(\td(\tilde{C_1} \times \tilde{C_2} )) - j_{ *} (\td (T)) .
\end{equation}

Indeed, the Todd class of the singular variety $C$ is defined in \cite[Chapter 18]{fulton} as $\tau_C (\mathcal{O}_C)$, for $\tau_C$ a group homomorphism from the $K$-theory of coherent sheaves on $C$ to the rational Chow group of $C$. Formula \eqref{two} follows by applying $\tau_C$ to the sequence of sheaves
\[
0 \to \mathcal{O}_C \to \nu_*(\mathcal{O}_{\tilde{C_1} \times \tilde{C_2}}) \to j_* \mathcal{O}_T \to 0,
\]
which is exact by the pushout construction.

 Applying the formulas for the Todd class of a product and of a blow-up of a $k$-smooth surface at a point, we compute
\begin{equation}\label{three}
\begin{cases} \td(T)= &   1 + (1-g+i) [\pt], \\  \td(\tilde{C}_1)= & 1 + (1-g+i, 1-g+i) - \frac{1}{2}  \sum_{k \in S^c \cup \{1 \} } E_k + (1-g+i)^2 [\pt], \\  \td(\tilde{C}_2) = & 1+ (1-g+i, 1-i) + (1-g+i)(1-i) [\pt]. \end{cases}
\end{equation}
(Here the $E_k$'s denote the exceptional fibers of the surface $\tilde{C_1}$.) For example, here is how one can proceed to calculate $\td(\tilde{C}_1)$ (the case of $\tilde{C}_2$ is similar). The surface $\tilde{C}_1$ is obtained as the blow-up of $T \times T$ at a set of closed points (see Construction~\ref{construction}). Using the multiplicativity of the Todd class, we deduce
\begin{equation} \label{uno}
\td(T \times T) = (1+ (1-g+i, 1)) \cdot (1+ (1, 1-g+i)) = 1 + (1-g+i, 1-g+i) + (1-g+i)^2 [\pt].
\end{equation}
For the degree $1$ component of the Todd class, we apply the formula for the first Chern class of the blow-up of a $k$-smooth surface
\begin{equation} \label{due}
\td_1(\tilde{C}_1) = \frac{c_1(\tilde{C}_1)}{2} = \frac{c_1(T \times T) -   \sum_{k \in S^c \cup \{1 \} } E_k}{2}.
\end{equation}
The degree $2$ component of the Todd class of a $k$-smooth surface remains unchanged after blowing up closed points, so we have
\begin{equation} \label{tre}
\td_2(\tilde{C}_1) = \td_2(T \times T),
\end{equation}
and our formula for $\td(\tilde{C}_1)$ in \eqref{three} is obtained by combining \eqref{uno}, \eqref{due} and \eqref{tre}.

 We now denote by $\nu_1$ and $\nu_2$  the restrictions of $\nu$ to the two components of the normalization $\tilde{C} = \tilde{C}_1 \sqcup \tilde{C}_2$. The Chern characters of the relevant line bundles on $\tilde{C}_1$ are
\begin{equation} \label{cc1}
\begin{cases}
\ch \nu_1^* \mathcal{O}(\mathcal{C}^+_{i,S})= & 1 + E_{1} - \frac{1}{2} [\pt], \\
\ch \nu_1^* \mathcal{O}(\mathcal{C}^-_{i,S})= & 1 + (1,0) - E_{1} - \frac{1}{2} [\pt];
\end{cases}
\end{equation}
and those on $\tilde{C}_2$ are
\begin{equation} \label{cc2}
\begin{cases}
\ch \nu_2^* \mathcal{O}(\mathcal{C}^+_{i,S})= &  1+ (1,0), \\
\ch \nu_2^* \mathcal{O}(\mathcal{C}^-_{i,S})= & 1.
\end{cases}
\end{equation}

 To calculate the degree on the left-hand side of \eqref{degreecomputation}, we compute
 \begin{align} \label{a} \deg \left( \left( \ch \nu_1^*F(\mathcal{C}^-_{i,S}) - \ch  \nu_1^* F \right)  \cap\td \tilde{C}_1 \right)&=   (1-g+i) -\frac{1}{2} + \deg(F_{|\mathcal{C}^-_{i,S}})-\frac{1}{2}\\  \nonumber &= \deg(F_{|\mathcal{C}^-_{i,S}}) -(g-i) , \end{align}
 \begin{equation} \label{b} \deg \left( \left( \ch \nu_2^*F(\mathcal{C}^-_{i,S}) - \ch  \nu_2^* F\right)  \cap\td \tilde{C}_2 \right)= 0, \end{equation}   \begin{equation} \label{c} \deg \left(\left( \ch j^*F(\mathcal{C}^-_{i,S})- \ch  j^* F\right) \cap\td T \right)=   0;  \end{equation}
where the last two expressions vanish because the curve $\mathcal{C}^-_{i,S}$ has empty intersection with $\nu_2(\tilde{C}_2)$ and with $j(T)$.

Altogether, taking $\eqref{a} + \eqref{b} - \eqref{c}$, we find that the degree on the left-hand side of \eqref{degreecomputation} equals $\deg\left( F_{| \mathcal{C}^-_{i,S}}\right) - (g-i)$. This concludes the proof of Lemma~\ref{Lemma: IntersectionCalculation}.
\end{proof}

\subsection{Wall-crossing of codimension $1$ cycles} \label{Sec: translated}

Here we explain how our description of the manner in which $\ThDivCl(\phi)$ depends on $\phi$ describes how $\Pic(\Jb_{g,n}(\phi))$ depends on $\phi$. We analyze $\Pic(\Jb_{g,n})$ by relating it to the Jacobian of the generic curve of $\Mb_{g,n}$, which we will denote by $J_{\eta}$.

\begin{lemma} \label{Lemma: PicJbarToPicEta}
	The restriction map $\Pic(\Jb_{g,n}(\phi)) \to \Pic(J_{\eta})$ induces an isomorphism
	\[
		\Pic( \Jb_{g,n}(\phi)) / \Pic( \Mb_{g,n}) \cong \Pic( J_{\eta}).
	\]
\end{lemma}
\begin{proof}
	Injectivity is a special case of the Seesaw theorem.  Surjectivity follows from the fact that $\Jb_{g,n}(\phi)$ is regular.  Indeed, because every line bundle can be represented as the difference of effective Cartier divisors, it is enough to show that every effective Cartier divisor on $J_{\eta}$ is the restriction of an effective Cartier divisor on $\Jb_{g,n}(\phi)$.  A given effective Cartier divisor $D$ is the restriction of its Zariski closure $\overline{D}$, and $\overline{D}$ is Cartier because $\Jb_{g,n}(\phi)$ is regular (hence locally factorial).
\end{proof}

\begin{lemma} \label{Lemma: NeronSeveri}
	The N\'{e}ron--Severi group $\operatorname{NS}(J_{\eta}) := \Pic(J_{\eta})/\Pic^{0}(J_{\eta})$ is freely generated by the class  $\theta_{\eta}$ of the theta divisor.
\end{lemma}
\begin{proof}
	The subset of curves with $\operatorname{NS}(J) \ne \mathbb{Z}$ is contained in a countable union of proper closed subsets \cite[Corollary~17.5.2]{bl}, so it does not contain the generic point $\eta$.  We conclude that $\operatorname{NS}(J_{\eta})=\mathbb{Z}$.  Writing $\theta_{\eta} = n \cdot \gamma$ for a generator $\gamma$, we have $\theta_{\eta}^g/g! = n^g \cdot \gamma^g/g!$, but $\gamma^g/g$ is an integer by the Riemann--Roch formula and $\theta_{\eta}^g/g!=1$, so $n = \pm 1$ and $\theta_{\eta}$ is a generator.
\end{proof}

\begin{definition}
	Given a fiberwise degree $0$ element $L \in \Pic( \Cb_{g,n})$ on the universal curve, we write $\operatorname{T}_{L}^{*}\ThDivCl(\phi)$ for the pullback of $\ThDivCl(\phi)$ by the rational map $\operatorname{T}_{L} \colon \Jb_{g,n}(\phi) \dashrightarrow \Jb_{g,n}(\phi)$ given by translation by $L$. We write $\operatorname{T}_{L}^*(\theta_{\eta}) \in \Pic( J_{\eta})$ for the restriction of $\operatorname{T}_{L}^{*}\ThDivCl(\phi)$ to the generic fiber.
\end{definition}
A priori the map $\operatorname{T}_{L}$ is only a rational map since the tensor product $F \otimes L$ with a $\phi$-stable sheaf $F$ can fail to be $\phi$-semistable. (One could show that this map in fact has no indeterminacy over $\Mm{g}{n}^{\text{TL}}$. This can be achieved by modifying $F \otimes L$ similarly to what we do in Section~\ref{PullbackTheta} for the problem of extending the section $s_{\vec{d}}$).

\begin{lemma} \label{Lemma: AlgEquivToZero}
	Every element of $\Pic^{0}(J_{\eta})$ equals $\operatorname{T}^{*}_{L}(\theta_{\eta})-\theta_{\eta}$ for some $L \in \Pic^{0}(C_{\eta})$
\end{lemma}
\begin{proof}
	The group $\Pic^{0}(J_{\eta})$ is the group of points of the dual abelian variety $J^{\vee}_{\eta}$. The homomorphism $J_{\eta} \to J^{\vee}_{\eta}$ defined by $L \mapsto \operatorname{T}^{*}_{L}(\theta_{\eta}) - \theta_{\eta}$ is an isomorphism because $\theta_{\eta}$ is a principal polarization. We deduce the desired result by taking the induced map on points.
\end{proof}

\begin{corollary} \label{Cor: translatedtheta}
	The group $\Pic(\Jb_{g,n}(\phi))$ is generated by $\pi^{*}\Pic(\Mb_{g,n})$ and by the elements of the form $\operatorname{T}^{*}_{L}(\ThDivCl(\phi))$ for $L$ a fiberwise degree $0$ element of $\Pic(\Cb_{g,n})$.
\end{corollary}
\begin{proof}
	Let  $M \in \Pic(\Jb_{g,n}(\phi))$ be given.  By Lemma~\ref{Lemma: NeronSeveri}, there exists an integer $k$ such that $k \cdot \ThDivCl(\phi)$ and $M$ have the same image in $\operatorname{NS}(J_{\eta})$.  The restriction of the difference $M- k \cdot \ThDivCl(\phi)$ to $J_{\eta}$ thus lies in $\Pic^{0}(J_{\eta})$, therefore it equals $\operatorname{T}^{*}_{L}(\theta_{\eta})-\theta_{\eta}$ for some $L \in \Pic^0(C_{\eta})$ by Lemma~\ref{Lemma: AlgEquivToZero}.  We conclude that
	\[
		M = (k-1) \cdot \ThDivCl(\phi) + \operatorname{T}_{L}^*(\ThDivCl(\phi)) + \pi^{*}(N) \text{ for some $N \in \Pic(\Mb_{g,n})$}
	\]
	by applying Lemma~\ref{Lemma: PicJbarToPicEta}.
\end{proof}

We use the following lemma to describe how a set of generators for $\Pic(\Jb_{g,n}(\phi))$ changes as $\phi$ changes.
\begin{corollary} \label{Cor: wctranslated}
Under the same hypotheses of Theorem~\ref{Thm: wc}, let $L$ be an element of $\Pic(\Cb_{g,n})$ of fiberwise degree $0$. Then
\begin{align}
\operatorname{T}_{L}^{*}(\ThDivCl(\phi_2)) - \operatorname{T}_{L}^{*}(\ThDivCl(\phi_1))  = &\left(\left\lfloor \phi_2^+(i,S)+ \frac{1}{2}\right\rfloor-i \right) \cdot \delta_{i,S} \notag \\   = & \left(d+1-i\right) \cdot \delta_{i,S} .
\end{align}
\end{corollary}
\begin{proof}
	Apply $\operatorname{T}_{L}^{*}$ to the wall-crossing Formula~\eqref{wc} in Theorem~\ref{Thm: wc} and observe that $\operatorname{T}_{L}^{*}$ acts as the identity on classes pulled back from $\Mb_{g,n}$.
\end{proof}

\section{Pullbacks of the theta divisor to $\overline{\mathcal{M}}_{g,n}$} \label{PullbackTheta}

In this section we study the pullback of the theta divisor to the moduli space of curves $\overline{\mathcal{M}}_{g,n}$ (or $\Mm{g}{n}^{\text{TL}}$), and compare our results with the existing literature. As in Section~\ref{Section: Wallcrossing}, we will assume that $n\geq1$ throughout. 

Let $\vec{d}=(d_1, \ldots, d_n)$ be a vector of integers such that $\sum_{j=1}^n d_j = g-1$. For any such vector the rule
\begin{equation}\label{muller} (\pi_T \colon C \to T, p_1, \ldots, p_n) \mapsto \left(\pi_T \colon C \to T, p_1, \ldots, p_n; \ \mathcal{O}_C(\mathcal{D})\right),\end{equation}
with $\mathcal{D}:=d_1 p_1 + \ldots + d_n p_n$, defines a section \begin{equation} \label{defsdopen} s_{\vec{d}} \colon \Mm{g}{n} \to \mathcal{J}_{g,n} \end{equation} of the forgetful map $\mathcal{J}_{g,n} \to \mathcal{M}_{g,n}$.

For $\phi_{\vec{d}}$ the nondegenerate parameter defined by  \begin{equation} \label{phimuller} \phi_{\vec{d}}(i,S)= (\phi^+(i,S), \phi^-(i,S)): = \left(\sum_{j \in S} d_j, g-1- \sum_{j \in S} d_j \right)=:(d_S, g-1- d_S) , \end{equation} the family of line bundles $\mathcal{O}_C(\mathcal{D})$  is fiberwise $\phi_{\vec{d}}$-stable and the rule \eqref{muller} defines a section $s_{\vec{d}}$ of the forgetful map $\Jb_{g,n}(\phi_{\vec{d}}) \to \mathcal{M}^{\text{TL}}_{g,n}$.

More generally, for any nondegenerate stability parameter $\phi \in V_{g,n}^{\text{TL}}$, we define the divisor
\begin{equation} \label{muller2}
\mathcal{D}(\phi):=d_1 p_1 + \ldots + d_n p_n + \sum_{(i,S)} \left( d_S -  \left\lfloor \phi^+(i,S) + \frac{1}{2}  \right\rfloor \right) \cdot \mathcal{C}_{i,S}^+.
\end{equation}
The family of line bundles $ \mathcal{O}_C(\mathcal{D}(\phi))$ is fiberwise $\phi$-stable by construction, and the rule \eqref{phimuller} defines a section $s_{\vec{d}} \colon \mathcal{M}^{\text{TL}}_{g,n} \to \Jb_{g,n}(\phi)$ of the forgetful map. This section is the unique one extending \eqref{defsdopen} because $\Jb_{g,n}(\phi)$ is separated.

In the following we compute the pullback by $s_{\vec{d}}$ of the theta class. Because $\overline{\mathcal{M}}_{g,n} \setminus \mathcal{M}_{g,n}^{\text{TL}}$ has codimension $2$,  the pullback to the treelike locus induces a well-defined group homomorphism $\pic(\Jb_{g,n}) \to \pic(\Mb_{g,n})$.

Recall that the integral Picard group of $\overline{\mathcal{M}}_{g,n}$ is generated (freely when $g \geq 3$) by the first Chern class of the Hodge bundle $\lambda$, the first Chern classes of the cotangent line bundles to the $j$-th marking $\psi_j$, the boundary strata classes $\delta_{i,S}$ and $\delta_{irr}$.

\begin{theorem} \label{pullback} The pullback of $\ThDivCl(\phi_{\vec{d}})$ from $\Jb_{g,n}( \phi_{\vec{d}})$ to $\overline{\mathcal{M}}_{g,n}$ is given by
\begin{equation} \label{Eq: easypullback}
s_{\vec{d}}^*\  \ThDivCl(\phi_{\vec{d}})= - \lambda + \sum_{j =1}^n {{d_j+1}\choose{2}} \cdot \psi_j.
\end{equation}
More generally, for any nondegenerate $\phi \in V_{g,n}^{\text{TL}}$, we obtain the equality
\begin{equation} \label{Eq: pullbackthetadiv}
s_{\vec{d}}^*\  \ThDivCl(\phi)= - \lambda + \sum_{j=1}^n {{d_j+1}\choose{2}} \cdot \psi_j + \sum_{(i,S)} \left( {{ \left\lfloor \phi^+(i,S)+ \frac{1}{2} \right\rfloor - i+1}\choose{2}}- {{ d_S- i+1}\choose{2}} \right) \cdot \delta_{i,S}.
\end{equation}
\end{theorem}

\begin{proof}

Assuming \eqref{Eq: easypullback} holds, Formula~\eqref{Eq: pullbackthetadiv} follows by the wall-crossing Formula \eqref{wc}. Indeed, let $\mathcal{Q}$ be any stability polytope such that $\phi_{\text{can}} \in \overline{\mathcal{Q}}$. The first summand in the coefficient of $\delta_{i,S}$  in \eqref{Eq: pullbackthetadiv} corresponds to subsequently crossing stability hyperplanes to reach $\mathcal{Q}$ from $\calP(\phi)$, and the second summand corresponds to subsequently crossing stability hyperplanes to reach $\calP(\phi)$ from $\calP(\phi_{\vec{d}})$ for $\phi_{\vec{d}}$ defined in Equation~\eqref{phimuller}.

We now prove equality~\eqref{Eq: easypullback}. We define $\mathcal{D}$ to be the effective divisor  $\sum_{j=1}^n d_j p_j$ in $\Cb_{g,n}$. As we  observed earlier, the line bundle $\mathcal{O}(\mathcal{D})$ is fiberwise $\phi_{\vec{d}}$-stable. We have
\begin{align}
s_{\vec{d}}^*\  \ThDivCl(\phi_{\vec{d}}) &= - s_{\vec{d}}^* \ c_1 \mathbb{R} \pi_* (F_{\text{tau}}) =   -c_1 ( \mathbb{R} \pi_* \mathcal{O}(\mathcal{D})) \nonumber\\
                           &= - \left[ \ch \mathbb{R} \pi_* ( \mathcal{O}(\mathcal{D})) \cap \td\left(\mathcal{M}_{g,n}^{\text{TL}}\right) \right]_{\codim=1} \nonumber \\
                           &=- \pi_* \left[ \ch  \mathcal{O}(\mathcal{D}) \cap \td \left(\Cb_{g,n} \right) \right]_{\codim=2} \nonumber \\
 \label{threeterms}            &=  \pi_* \left[- \frac{\mathcal{D}^2}{2} + \mathcal{D} \cdot \frac{K_{\Cb_{g,n}}}{2}  - \td_2 ({\Cb}_{g,n})\right],
\end{align}
where we applied the theorem on cohomology and base change, the definition of theta divisor, the fact that $\ch_0 (\mathbb{R} \pi_* \mathcal{O}(\mathcal{D}))$ equals zero, and the Grothendieck--Riemann--Roch formula for stacks (see e.g.~\cite[Theorem 3.5]{edidinb}).

The first term in \eqref{threeterms} equals
\begin{equation} \label{todd1}
- \pi_*\left( \frac{\mathcal{D}^2}{2}\right) = \frac{1}{2} \sum_{j=1}^n d_j^2 \psi_j,
\end{equation}
because two different sections $p_j$ and $p_k$ are by definition disjoint, and by the very definition of the $\psi$-classes:
\[
\psi_j := - \pi_* ( p_j^2 ).
\]

To compute the second and third terms in \eqref{threeterms}, we identify the universal curve $\Cb_{g,n}$ with $\Mb_{g,n+1}$. The canonical class equals
\[
K=K_{\overline{\mathcal{M}}_{g,n+1}} = 13 \lambda  + \psi -2 \delta, \textrm{ where } \delta:= \delta_{irr} + \sum_{(i,S)} \delta_{i,S} \textrm{ and }  \psi:=\sum_{j=1}^{n+1} \psi_j.
\]
Using the pushforward formulas (that can be derived from \cite[Theorem~2.8]{logan})
\begin{align}
\pi_* ( p_j \cdot \lambda) &= \ \lambda,\nonumber \\ \pi_* (p_j \cdot \psi_k)&= \ \begin{cases} 0 & \textrm{ when } j=k, \\ \psi_k & \textrm{ when } j \neq k, \end{cases} \nonumber \\ \pi_*(\psi_j \cdot \delta_{irr}) &= \ \delta_{irr}, \nonumber \\ \pi_* ( p_j \cdot \delta_{i,S}) &= \ \begin{cases} \delta_{i,S} & \textrm{ when } \{p_j, p_{n+1}\} \subseteq S \textrm{ or } \{p_j, p_{n+1}\} \subseteq S^c, \\ 0 & \textrm{ otherwise,} \end{cases} \nonumber
\end{align}
(where $\delta_{0,\{j\}}$ is interpreted as $- \psi_j$ in the last formula), the second term in \eqref{threeterms} becomes
\begin{equation} \label{todd2}
\pi_*\left([\mathcal{D}] \cdot \frac{K_{\Cb_{g,n}}}{2} \right) = \frac{13}{2} (g-1) \cdot \lambda + \sum_{j=1}^n \frac{g-1+d_j}{2} \cdot  \psi_j -  (g-1)\cdot \delta.
\end{equation}

Finally, the third term equals
\begin{equation} \label{todd3}
- \pi_* (  \td_2  ( \Cb_{g,n})  ) = - \left( \frac{g-1}{2}\cdot ( 13 \lambda + \psi - 2 \delta) + \lambda \right).
\end{equation}
Indeed $\td_2=\frac{K^2+c_2}{12}$, and we read the formula for $c_2$ in \cite[page~765]{bini}. (Note that the formula for the second Chern class appears with an error in the coefficient of $\kappa_2$, which should be $-\frac{1}{2}$. This can be quickly checked by applying the Grothendieck--Riemann--Roch formula to the sheaf $\omega_{\pi}^{\otimes 2}(p_1 + \ldots + p_n)$ along the universal curve $\pi \colon \C_{g,n} \to \mathcal{M}_{g,n}$.) The pushforward \eqref{todd3} can then be computed with the aid of the pushforward formulas (that again can be derived from \cite[Theorem~2.8]{logan})
\begin{align}
\pi_*(K^2) = \pi_* ((\pi^*K + \omega_{\pi})^2) = & \ 2 \cdot \pi_* (\omega_{\pi})  \cdot K +  \pi_*( \omega_{\pi}^2) \nonumber \\
  = & \ 2 \cdot (2g-2) \cdot ( 13 \lambda + \psi - 2 \delta) + 12 \lambda - \delta, \nonumber \\
\pi_* (\kappa_2) =& \ 12 \lambda + \psi - \delta,\nonumber \\
\pi_* (\xi_{irr *} (\psi+\psi)) =&\ 2 \cdot \delta_{irr},\nonumber  \\
\pi_* (\xi_{i,S *} (1 \otimes \psi + \psi \otimes 1)) =&\ \delta_{i,S}, \nonumber \\
\pi_* (\xi_{0,\{j,n+1\} *} (1 \otimes \psi + \psi \otimes 1)) =&\ \psi_j. \nonumber
\end{align}
(Following the notation from \cite{bini}, here $\xi_{irr}$ and $\xi_{i,S}$ are the gluing maps, and $\kappa_2$ is the Arbarello--Cornalba kappa class).

Plugging the three terms \eqref{todd1} \eqref{todd2} and \eqref{todd3} in Equation~\eqref{threeterms},  we deduce \eqref{Eq: easypullback}.
\end{proof}

We now compare our results with certain pullbacks of similar theta divisors that have recently been studied by different authors.

\subsection{The class of Hain} \label{hgz}
Hain  studied a problem similar to the problem of computing $s_{\vec{d}}^{*}(\ThDivCl(\phi))$.  He constructed a theta divisor on the moduli space of multidegree $0$ line bundles on compact type curves.  His result is different from the results of this paper in two ways.  First, his construction is different. Hain's construction involves a choice of theta characteristic, uses the formalism of theta functions, and produces a $\mathbb{Q}$-divisor class on a moduli space of degree $0$ line bundles (rather than a moduli space of degree $g-1$ line bundles) \cite[Section~11.2, page~561]{hain13}.  Second, the pullback of the resulting divisor class differs from the pullbacks of the $\theta(\phi)$'s constructed in this paper.  Indeed, in \cite[Theorem~11.7]{hain13},  Hain computed the pullback of $\theta_{\alpha}$ by $s_{\vec{d}}$ as:
\begin{align} \label{Eqn: HClass}
	[\overline{D}_{\vec{d}}(\text{Ha})] 	=& - \lambda + \sum_{j=1}^n {{d_j+1}\choose{2}}\cdot  \psi_j - \sum_{(i,S)} {{d_S -i +1} \choose {2}}\cdot  \delta_{i,S} + \frac{\delta_{{irr}}}{8} \\
								=& [\overline{D}_{\vec{d}}(\phi_0)]+\frac{\delta_{{irr}}}{8},
\end{align}
and being a nonintegral Chow class, this never equals  $s_{\vec{d}}^{*}(\ThDivCl(\phi))$.

The results of this paper illuminate some of the structure of  \eqref{Eqn: HClass}.  The term $\lambda + \sum_{j=1}^n {{d_j+1}\choose{2}} \cdot \psi_j$ is $[\overline{D}_{\vec{d}}(\phi_{\vec{d}})]$, while the term $\sum {{d_S -i +1} \choose {2}} \cdot \delta_{i,S}$ is a wall-crossing term, the difference between $[\overline{D}_{\vec{d}}(\phi_{\vec{d}})]$ and $[\overline{D}_{\vec{d}}(\phi_0)]$ described by Theorem~\ref{Thm: wc}.

Finally, a caution to the reader.  Grushevsky--Zakharov gave an alternative proof of~\eqref{Eqn: HClass} in  \cite[Theorem~2, Equation~(3.4)]{grushevsky14a}), their definition of the theta divisor in  \cite{grushevsky14a} is different from the definition in \cite{hain13}.  Over the locus of compact type curves, the theta divisor is defined on \cite[page~4053, second paragraph]{grushevsky14a} to be the image of an Abel (or sum) map out of the family of symmetric powers over the moduli space of compact type curves.  It is significant that this is taken as the definition over the locus of compact type curves and not over all of $\Mb_{g,n}$.  As a map defined over  $\Mb_{g, n}$, the Abel map has indeterminacy, but one can still define its image as the projection  of the Zariski closure of the graph.  This construction produces an integral Chow class, and as such, it cannot equal Hain's class $[\overline{D}_{\vec{d}}(\text{Ha})]$ (which is nonintegral, as it evidently appears in Formula~\eqref{Eqn: HClass}).

\subsection{The stable pairs class} \label{Subsect: StablePair}
In the introduction we mentioned the divisor $[\overline{D}_{\vec{d}}(\text{SP})]$ that is the pullback of the theta divisor of the unique family of stable semiabelic (or quasiabelian) pairs extending the principally polarized universal Jacobian.  Here we describe this extension in greater detail.

Recall that a stable semiabelic pair is a pair $(\overline{P}, D)$ consisting of a (possibly reducible) seminormal projective variety $\overline{P}$ with a suitable action of a semiabelian variety $G$ together with an ample effective divisor $D \subset \overline{P}$ that does not contain a $G$-orbit \cite[Definition~1.1.9]{alexeevb}.  Stable semiabelic pairs satisfy a stable reduction theorem \cite[Theorem~5.7.1]{alexeevb} that implies there is, up to isomorphism of pairs, at most one extension of the family of principally polarized Jacobians $(\J_{g,n}/\Mm{g}{n}, \Theta)$ to a family of stable semiabelic pairs $(\Jb_{g,n}/\Mm{g}{n}^{\text{TL}}, \ThDiv)$.

For $n=0$ (a case not studied here), Alexeev has proven that this unique extension  exists and is realized by the Caporaso--Pandharipande family, the family of compactified Jacobians associated to the degenerate parameter $\phi_{\text{can}}$ \cite[Theorem~5.1, Theorem~5.3, Corollary~5.4]{alexeev04}.  For $n>0$, the unique extension $(\overline{\J}_{g,n}/\Mb_{g,n}, \ThDiv)$ of $\J_{g,n}$ is the pullback of $(\Jb_{g,0}, \ThDiv_{g,0})$ by the forgetful morphism $\Mb_{g,n} \to \Mb_g$.

An alternative description of this extension is provided by the following lemma:
\begin{lemma} \label{Lemma: FamilyIsSemiabelic}
	If $\phi_0 \in V_{g,n}^{\text{TL}}$ is nondegenerate and such that $\phi_{\text{can}} \in \overline{\calP}(\phi_0)$, then the restriction of the pair  $(\Jb_{g,n}(\phi)/\Mm{g}{n}^{\text{TL}}, \ThDiv(\phi))$ to the open substack $\Mb_{g,n}^{\leq 1} \subseteq \Mm{g}{n}^{\text{TL}}$ of stable curves with at most one node is a stable semiabelic pair.
\end{lemma}
\begin{proof}
	For convenience, call $\mathcal{U}:= \Mb_{g,n}^{\leq 1}$. The main point to prove is that a fiber of $\ThDiv(\phi_0)|\mathcal{U} \to \mathcal{U}$ is ample and does not contain an orbit of the action of the multidegree $0$ Jacobian, and we prove this by directly computing the theta divisor, which has a particularly simple structure.  To begin, observe that both $\Jb_{g,n}(\phi_0)|\mathcal{U} \to \mathcal{U}$ and $\ThDiv(\phi_0)|\mathcal{U} \to \mathcal{U}$ are flat by Corollary~\ref{Cor: JbExists} and Lemma~\ref{Lemma: WhenIsPhiFlat?}, so it is enough to fix a pointed curve $(C, p_1, \dots, p_n) \in \mathcal{U}$  and prove that the fiber  $\Jb_{C}$ and the effective divisor $\ThDiv_{C}$ form a stable semiabelic variety.
	
	Alexeev has proved quite generally that the compactified Jacobian of a nodal curve is a stable semiabelic variety \cite[Theorem~5.1]{alexeev04}, so to prove the specific pair   $(\Jb_{C}, \ThDiv_{C})$ is a stable pair, we need to prove that $\ThDiv_{C}$ is ample and does not contain an orbit of the action of the moduli space $\J^0_{C}$ of multidegree $0$ line bundles.  There are two cases to consider: when $C$ is irreducible and when $C$ is reducible.
	
	When $C$ is irreducible, $\ThDiv_{C}$ is ample by \cite[Corollary~14]{soucaris} and does not contain a group orbit by the proof of  \cite[Proposition~7]{soucaris}.  When $C$ is reducible, $C$ must have two irreducible components, $C^{+}$ and $C^{-}$, and the computation from Example~\ref{Example: TwoComponentExample} shows that the $\phi_0$-stable sheaves are either the line bundles of bidegree $(g^{+}-1, g^{-})$ or the line bundles of bidegree $(g^{+}, g^{-}-1)$.  In the first case, restricting to components defines an isomorphism $\Jb_{C}(\phi_0) \cong \J_{C^{+}}^{g^{+}-1} \times \J_{C^{-}}^{g^{-}}$ that identifies $\ThDiv(\phi_0)$ with $p_{2}^{*}(\text{node} + \ThDiv_{C^{+}})+p_2^{*}(\ThDiv_{C^{-}})$.  (Here $p_1, p_2$ are the projection morphisms).  This identifies $(\Jb_{C}(\phi_0), \ThDiv_{C})$ as the product of principally polarized varieties, and such a  product satisfies the desired conditions.  The case of bidegree $(g^{+}, g^{-}-1)$ is entirely analogous, with the roles of $C^{+}$ and $C^{-}$ being switched.
\end{proof}

\begin{remark}
	Observe that Lemma~\ref{Lemma: FamilyIsSemiabelic} implies that the unique extension of $(\J_{g,n}, \ThDiv)$ to a family of stable pairs over  $\Mb_{g,n}^{\leq 1}$ admits multiple descriptions as a moduli space.  The authors expect that this remains true over $\Mm{g}{n}^{\text{TL}}$ but, as our goal is to establish Equation~\eqref{Eqn: PairOne} from the Introduction, we do not pursue this issue here.
\end{remark}

An immediate consequence is that Equation~\eqref{Eqn: PairOne} holds.
\begin{corollary} \label{Cor: StablePair}	
	We have the following equality of Chow classes:
\begin{align}
		[\overline{D}_{\vec{d}}(\text{SP})] 	=& [\overline{D}_{\vec{d}}(\phi)] \quad \text{for any nondegenerate $\phi$ such that $\phi_{\text{can}} \in \overline{\mathcal{P}}(\phi)$}   \\
									=& - \lambda + \sum_{j=1}^n {{d_j+1}\choose{2}}\cdot  \psi_j - \sum_{(i,S)} {{d_S -i +1} \choose {2}}\cdot  \delta_{i,S}. \label{Eqn: StablePairTwo}
	\end{align}
\end{corollary}
\begin{proof}
	By Lemma~\ref{Lemma: FamilyIsSemiabelic} we have $[\overline{D}_{\vec{d}}(\text{SP})] = [\overline{D}_{\vec{d}}(\phi_0)]$ for any $\phi_0$ such that $\phi_{\text{can}} \in \overline{\calP}(\phi_0)$.  The other equality follows from Theorem~\ref{pullback}.
\end{proof}

An important case of Corollary~\ref{Cor: StablePair} was proven by Dudin \cite{dudin}.  In \cite[Section~4.3]{dudin}, Dudin computes the pullback of the theta divisor on one of the moduli spaces of quasi-stable sheaves studied by Melo in \cite{melo16}.  In the notation of this paper, the moduli space Dudin studies is  $\Jb_{g, n}(\phi_{\epsilon})$ for
\begin{gather*}
	\phi_{\epsilon} := \phi_{\text{can}} + \epsilon \cdot \chi \text{ for $0 < \epsilon << 1$, with} \\
	\chi(\Gamma)(v) := \begin{cases} \#\operatorname{Vert}(\Gamma)-1 & \text{if $p(1) = v$;} \\ -1 & \text{otherwise.} \end{cases}
\end{gather*}
(To see that $\phi_{\epsilon}$-stability coincides with the quasi-stability condition studied by Dudin, observe that $\phi_{\epsilon}$ is defined so that the semistability inequality \eqref{Eqn: DefOfStability} holds for $\phi=\phi_{\text{can}}$ and the inequality is strict when the subgraph $\Gamma_0$ contains the first marking $1$.)

Using techniques similar to the ones used in this paper, Dudin proves that the pullback of $\ThDivCl(\phi_{\epsilon})$ under the section $s_{\vec{d}}$ is the class in Equation~\eqref{Eqn: StablePairTwo}.

\subsection{The class of M\"uller} \label{Subsect: Mueller}
M\"{u}ller studied a different extension of $[D_{\vec{d}}]$ in \cite{mueller13}.  Assuming that at least one $d_j$ is negative, M\"uller defined $\overline{D}_{\vec{d}}(\text{M\"{u}}) \subset \Mb_{g,n}$ to be the Zariski closure in $\Mb_{g,n}$ of \[D_{\vec{d}}:= \{(C, p_1, \ldots, p_n) \in \Mm{g}{n} \colon \ h^0(C, \mathcal{O}_C(d_1 p_1+ \ldots+ d_n p_n) \}\subset \Mm{g}{n}. \]

M\"uller implicitly works with a convention for the indices of the boundary divisors $\Delta_{i,S}$ that is different from the one we have fixed in Definition~\ref{Definition: AdmissiblePair}. This convention depends on $\vec{d}$. We let $S^{+}:=\{j \in \{1,\ldots ,n\} \colon d_{j}>0\}$, and we construct a set of indices $\mathcal{I}:=\{(i,S)\}$ by first including all stable pairs $(i,S)$ such that $0 \leq i \leq g$ and $S \subseteq S^{+}$ and then completing it to a full set of representatives for the equivalence relation $(i,S) \sim (g-i, S^{c})$ on the subset of $\{0, \ldots , g \} \times \{S \subseteq [n]\}$ of stable pairs (a pair $(i,S)$ is unstable if $i=0$ and $|S| < 2$ or if $i=g$ and $|S^c| < 2$, and it is stable otherwise).

M\"uller computed in \cite[Theorem~5.6]{mueller13} the class of the Zariski closure:
\begin{equation} \label{Eqn: MuellerDivisor}
	[\overline{D}_{\vec{d}}(\text{M\"{u}})] = - \lambda + \sum_{j=1}^n {{d_j+1}\choose{2}} \cdot  \psi_j - \sum_{\substack{(i,S) \in \mathcal{I} \\ S \subseteq S^+}}  {{ |d_{S}-i|+1 }\choose{2}} \cdot \delta_{i,S} - \sum_{\substack{(i,S) \in \mathcal{I} \\ S \not\subseteq S^+}} {{ d_{S}- i+1}\choose{2}} \cdot \delta_{i,S}
\end{equation} and Grushevsky--Zakharov gave  in \cite[Theorem~2]{grushevsky14a} an alternative proof of this formula. (Note that the set of indices $\mathcal{I}$ depends on a choice, but Formula~\eqref{Eqn: MuellerDivisor} is independent of that choice).

Comparing \eqref{Eq: pullbackthetadiv} with \eqref{Eqn: MuellerDivisor}, we see that if $\phi_0$ is such that $\phi_{\text{can}} \in \overline{\calP}(\phi_0)$, then
\begin{equation} \label{differencemuller}
	[\overline{D}_{\vec{d}}(\phi_0)] = [\overline{D}_{\vec{d}}(\text{M\"{u}})] + \sum_{(i,S) \in T_{\vec{d}}} (i-d_{S}) \cdot \delta_{i,S},
\end{equation}
where $T_{\vec{d}}$ is defined by
\[T_{\vec{d}}:=\{(i,S) \in \mathcal{I} \colon  \ d_j>0 \textrm{ for all }j \in S, \textrm{ and }d_S <i\}.\]
(Observe that $T_{\vec{d}}$ is the set of indices $(i,S)$ such that $s_{\vec{d}}(\Delta_{i,S})$ is completely contained in $\ThDiv(\phi_0) \subset \Jb_{g,n}(\phi_0)$.)

Inspecting Equation~\eqref{differencemuller}, we see that the divisor classes $[\overline{D}_{\vec{d}}(\phi_0)]$ and $[\overline{D}_{\vec{d}}(\text{M\"{u}})]$ are equal if and only if $T_{\vec{d}}=\emptyset$.  Thus from Lemma \ref{Lemma: inclusiondivisors} we deduce the following.

\begin{corollary} \label{cor: mullerequalstheta} Let $\phi_0$ be any nondegenerate element of $V_{g,n}^{\text{TL}}$ such that $\phi_{\text{can}} \in \overline{\calP}(\phi_0)$. The inclusion of the closed substack $\overline{D}_{\vec{d}}(\text{M\"{u}})$ in $\overline{D}_{\vec{d}}(\phi)$ is an isomorphism if and only if $\phi= \phi_0$ and $T_{\vec{d}}= \emptyset$.
\end{corollary}

\section{Acknowledgments}
Jesse Leo Kass was partially sponsored by the Simons Foundation under Grant Number \#429929  and by the National Security Agency under Grant Number H98230-15-1-0264. The United States Government is authorized to reproduce and distribute reprints notwithstanding any copyright notation herein. This manuscript is submitted for publication with the understanding that the United States Government is authorized to reproduce and distribute reprints. Nicola Pagani was supported by the EPSRC First Grant EP/P004881/1 with title  ``Wall-crossing on universal compactified Jacobians''.

The authors would like to thank the referee for suggestions that significantly improved the exposition.

This project was first conceived when the two authors were working for the Institut f\"{u}r Algebraische Geometrie of the Leibniz Universit\"{a}t Hannover, and the authors have benefitted from the hospitality of the IAG during the summer of 2014. The authors are particularly thankful to Klaus Hulek for his academic guidance and for his financial support.

The authors are grateful to Orsola Tommasi for discussions that led them to consider translations of the theta divisor and to Barbara Fantechi and Nicola Tarasca for useful discussions related to the development of this paper.

\bibliography{wallcrossing}
\end{document}